\documentclass[12pt]{amsart}
\usepackage{appendix}
\usepackage{geometry}
\usepackage{amssymb}
\usepackage{amsmath}
\usepackage{amsthm}
\usepackage{mathtools}
\usepackage{mathrsfs}
\usepackage{enumitem}
\usepackage{tikz}
\usepackage{tikz-cd}
\usepackage{graphicx}
\DeclareFontFamily{U}{rsfs}{\skewchar\font127 }
\DeclareFontShape{U}{rsfs}{m}{n}{%
   <-6> rsfs5
   <6-8> rsfs7
   <8-> rsfs10
}{}
\usepackage{bm}
\usepackage{hyperref}
\hypersetup{
	colorlinks=true,
	linkcolor=cyan!50!blue,
	filecolor=blue,      
	urlcolor=red,
	citecolor=green,
}
\usepackage{cleveref}
\usepackage{makecell}
\usepackage{framed}
\usepackage{graphicx}
\usepackage{xcolor}
\usepackage{listings}
\usepackage{multirow}
\usepackage{multicol}
\usepackage{indentfirst}
\usepackage{booktabs}
\usepackage{tikz}
\usepackage{tikz-3dplot}
\usepackage{tikz-cd}
\usepackage{float}
\usepackage{algorithm}
\usepackage{algorithmicx}
\usepackage{algpseudocode}
\usepackage{subfigure}
\usepackage{lineno}
\usepackage{xypic}
\usepackage{setspace}
\usepackage{marginnote}

\usetikzlibrary{patterns, arrows.meta, calc}

\newcommand*{\be}[1]{\begin{equation}\label{#1}}
\newcommand*{\ee}{\end{equation}}

\newtheorem{theorem}{Theorem}[section]

\newtheorem{lemma}{Lemma}[section]
\newtheorem{proposition}{Proposition}[section]
\newtheorem{definition}{Definition}[section]

\theoremstyle{remark}
\newtheorem{example}{Example}[section]
\newtheorem{remark}{Remark}[section]

\definecolor{pink}{RGB}{255,45,115}

\DeclareMathOperator{\grad}{grad}
\DeclareMathOperator{\hess}{hess}
\DeclareMathOperator{\curl}{curl}
\DeclareMathOperator{\inc}{inc}

\DeclareMathOperator{\dev}{dev}
\DeclareMathOperator{\sym}{sym}
\DeclareMathOperator{\diverenge}{div}
\DeclareMathOperator{\im}{im}
\DeclareMathOperator{\tr}{tr}

\DeclareMathOperator{\st}{star}

\renewcommand{\div}{\diverenge}
\renewcommand{\emptyset}{\varnothing}
\usepackage{geometry}
\theoremstyle{remark}

\numberwithin{equation}{section}

\DeclareMathOperator{\rot}{rot}

\title[Finite element complexes with trace structures]{Finite Element Complexes with Trace Structures: A unified framework for cohomology and bounded interpolation}
\author{Jun Hu}
\address{LMAM and School of Mathematical Sciences, Peking University, Beijing 100871, People’s Republic of China}
\email{hujun@math.pku.edu.cn}

\author{Yizhou Liang}
\address{Mathematical Institute, University of Oxford,
Andrew Wiles Building, Woodstock Road,
Oxford OX2 6GG, UK}
\email{yizhou.liang@maths.ox.ac.uk}

\author{Ting Lin}
\address{LMAM and School of Mathematical Sciences, Peking University, Beijing 100871, People’s Republic of China}
\email{lintingsms@pku.edu.cn}

\thanks{The first author was supported by the Fundamental and Interdisciplinary Disciplines Breakthrough Plan of the Ministry of Education of China (JYB2025XDXM204), the National Key R\&D Program of China (Grant No.~2023YFA1011500), and the National Natural Science Foundation of China (Grant No.~12288101). The second author was partly supported through a Royal Society University Research Fellowship (URF\textbackslash R1\textbackslash 221398, RF\textbackslash ERE\textbackslash 221047).}


\begin{document}
\maketitle 
\begin{abstract}
This paper develops a unified framework for studying the cohomology and constructing $L^2$-bounded commuting interpolation operators for finite element complexes, such as the Hessian, divdiv, elasticity and Stokes complexes. Compared with the standard de Rham complex, these complexes present additional challenges arising from algebraic constraints on tensor fields, such as symmetry and tracelessness, as well as additional smoothness requirements. The main ingredient of our framework is a direct sum decomposition of finite element complexes into skeletal and bubble parts. To establish this decomposition, we introduce two key structures: trace structures and generalized currents. In this decomposition, the bubble part is characterized by bubble complexes associated with subsimplices of different dimensions, including vertices, edges, and faces. We prove that the exactness of the resulting bubble complexes ensures that the finite element complex has the correct cohomology. Based on this decomposition, we further construct $L^2$-bounded interpolation operators for these finite element complexes.
\end{abstract}

\section{Introduction}
Many partial differential equations are naturally associated with differential complexes, whose structures play an important role in their analysis. Preserving these structures at the discrete level is crucial for the stability of finite element methods. A fundamental example is the de~Rham complex, whose finite element discretizations have been extensively studied in the framework of finite element exterior calculus; see, e.g., \cite{arnold2006finite,arnold2010finite,arnold2018finite} and the references therein. Other important differential complexes include the Hessian, divdiv, elasticity and Stokes complexes, as well as the conformal deformation and conformal Hessian complexes \cite{arnold2021complexes,pauly2020divdiv,neilan2015discrete,arnold2002differential}. At the continuous level, these complexes can be derived from suitable de~Rham complexes through the Bernstein-Gelfand-Gelfand (BGG) construction, which also relates their cohomology to that of the underlying complexes \cite{arnold2021complexes}. These homological relations have recently been explained in terms of short exact sequences of complexes and the associated long exact cohomology sequences \cite{christiansen2026note}. At the discrete level, finite element constructions include the Hessian complex \cite{hu2021conforming}, the divdiv complex \cite{chen2022finitedivdiv,hu2022conforming,hu2024family,hu2024finite}, the elasticity complex \cite{christiansen2018nodal,chen2022finiteelasticity,hu2024finite}, and the Stokes complex \cite{neilan2015discrete,hu2022stokes}. Complexes related to numerical relativity and the Cotton--York operator have also been studied \cite{hu2023finiteelementssymmetrictraceles,huang2025conformal,guo2025conformal}. Systematic constructions based on smooth finite elements and discrete BGG diagrams have also been developed \cite{chen2026complex,hu2025finitevalue}.

A major challenge in constructing these finite element complexes is to design discrete spaces that satisfy the required algebraic constraints, such as symmetry and tracelessness conditions, and form a subcomplex of the continuous complex. In addition, many existing constructions impose additional smoothness beyond that
required by their ambient Sobolev spaces, often with different degrees of smoothness at vertices, edges, and faces. Beyond constructing a conforming subcomplex, it is also important to determine its cohomology and construct suitable bounded commuting projections. However, existing studies on the cohomology of these complexes have mainly focused on topologically trivial domains. This raises a natural question:  \emph{do these discrete complexes preserve the cohomology of the underlying continuous complexes on topologically nontrivial domains?} The cohomology of finite element de Rham complexes on topologically nontrivial domains has been studied without boundary conditions \cite{arnold2006finite}, with homogeneous boundary conditions \cite{christiansen2008smoothed}, and with partial boundary conditions \cite{licht2017complexes}. Related cohomology results have been established for compatible finite element systems (FES) \cite{christiansen2011topics,christiansen2023finite,christiansen2018generalized} and discrete de Rham methods \cite{dipietro2023cohomology,christiansen2025interpretation}. Beyond the de Rham setting, recent work has characterized the cohomology of distributional Hessian and divdiv complexes \cite{hu2025distributional}, Regge complexes and their extensions \cite{christiansen2026extended,christiansen2026regge}, and polygonal Stokes and BGG complexes \cite{dipietro2025design}. 

Another key ingredient in the analysis is the construction of bounded commuting projections, which, unlike canonical finite element interpolants, are well-defined and bounded on the natural spaces of the underlying Hilbert complex. For the de Rham complex, such projections have been developed using smoothing techniques, including variants that accommodate full or partial homogeneous essential boundary conditions \cite{christiansen2008smoothed,licht2019smoothed}. Of particular interest are the local commuting projections of Falk and Winther \cite{falk2014local,falk2015double}, which are bounded in the natural graph norms, and the local $L^2$-bounded commuting projections of Arnold and Guzm\'{a}n \cite{arnold2021local}. More recent constructions include local $L^2$-bounded projections based on discrete local problems \cite{ern2025local} and local graph-norm-bounded projections preserving discrete traces \cite{ern2026bounded}. Bounded commuting projections are fundamental to the analysis of mixed finite element methods \cite{brezzi1974existence,brezzi1991mixed}, in particular for deriving optimal error estimates for mixed formulations of the Hodge Laplacian \cite{arnold2006finite} and for eigenvalue problems \cite{boffi2010finite}.

This paper proposes a unified framework, called \emph{finite element complexes with trace structures} (FECTS), to study the cohomology and local commuting projections of conforming finite element complexes with additional smoothness. The key step is to construct a suitable set of finite element degrees of freedom and use it to obtain a direct sum decomposition of the entire complex. The resulting decomposition consists of exact local bubble subcomplexes associated with subsimplices and a skeletal subcomplex carrying the cohomology. The same splitting is then used to construct local bounded commuting projections.

More precisely, let $\mathcal T$ be a simplicial mesh of $\Omega$, and consider the conforming finite element complex (e.g., Hessian, divdiv, elasticity, or Stokes complex)
\begin{equation*}
\begin{tikzcd}
0 \ar[r,""] & \bm A^0 \ar[r,"d^0"] & \cdots  \ar[r,"d^{n-1}"] & \bm A^n \ar[r,""] & 0.
\end{tikzcd}
\end{equation*}
Under the direct sum decomposition and compatibility assumptions stated below, we establish
$$
\bm A^\bullet =  \bm S^\bullet(\mathcal{T}) \oplus \bigoplus_{\tau \in \mathcal T} \mathbb B^\bullet(\tau),\quad \bm S^{\bullet}(\mathcal{T}) \cong C^{\bullet}(\mathcal{T};\mathcal{Z}),$$
where $\mathcal Z$ is the coefficient space associated with the underlying continuous complex. This is a decomposition into subcomplexes, rather than a degreewise decomposition of finite element spaces: the differential preserves each summand, and the associated inclusions and projections commute with the differential. The skeletal subcomplex $\bm S^\bullet(\mathcal T)$ is isomorphic to the simplicial cochain complex with coefficients in $\mathcal Z$. Each bubble subcomplex $\mathbb B^\bullet(\tau)$ is exact and supported on the elements containing $\tau$, with $\tau$ ranging over subsimplices of every dimension. Thus, the splitting separates the part representing the mesh topology from exact contributions associated with vertices, edges, faces, and element interiors.

\paragraph{\emph{Trace structures}} (\Cref{sec:ts}) To construct these subcomplexes, we introduce trace structures that encode the values and derivatives whose continuity is prescribed on each subsimplex. They define the local bubble spaces and, together with the geometric decomposition assumption, organize the finite element degrees of freedom and their local extensions. Whereas the restriction maps in FES and its generalizations compose exactly \cite{christiansen2011topics}, our trace structures are not required to satisfy
\begin{equation*}
    \tr_{\tau \to \eta} \circ \tr_{\sigma \to \tau} = \tr_{\sigma\to\eta},\qquad \forall\eta\trianglelefteq\tau 
    \trianglelefteq\sigma.
\end{equation*}
This flexibility accommodates continuity conditions involving different derivatives on different subsimplices. For example, a direct element-to-vertex trace may retain derivatives that cannot be recovered from an intermediate edge trace. The framework therefore works directly with the prescribed continuity conditions of finite elements, without requiring their completion to a strictly compositional system of traces.

\paragraph{\emph{Generalized currents}} (\Cref{sec:gen-cur}) Generalized currents provide the second ingredient in the decomposition. They are simplex-associated functionals satisfying a generalized Stokes identity and identify the skeletal degrees of freedom that carry the global cohomological information. After separating these degrees of freedom, the remaining part of the complex consists of reduced bubble complexes associated with subsimplices. The exactness of these bubble complexes ensures that the associated local subcomplexes contribute no additional cohomology. Related simplex-supported distributions have been used in the cohomological analysis of distributional Hessian, divdiv, and Regge complexes \cite{hu2025distributional,christiansen2026extended}; here they are incorporated into the geometric decomposition of conforming finite element complexes.

The decomposition consequently identifies the cohomology of $\bm A^\bullet$ with $\mathcal Z\otimes \mathcal{H}_{dR}(\Omega)$, without assuming that $\Omega$ has trivial topology. Its significance also extends beyond this cohomological conclusion: the local extensions and the separation of skeletal and bubble degrees of freedom provide the structure for constructing local commuting projections. Under the additional local analytical and scaling assumptions stated below, these projections are uniformly $L^2$-bounded. 

We illustrate the framework with several two-dimensional examples, including the Hermite--Stenberg complex \cite{christiansen2018nodal}, the Falk--Neilan Stokes complex \cite{falk2013stokes}, and the Hessian and divdiv complexes \cite{christiansen2018nodal,hu2021family}, as well as the three-dimensional Hu--Liang Hessian complex \cite{hu2021conforming}. For each example, we identify the trace structures, the reduced bubble complexes in the relevant dimensions, and the generalized currents through which the abstract results apply.

The paper is organized as follows. \Cref{sec:main-results} introduces the framework and states the main results. \Cref{sec:proofs} develops the direct sum decomposition of the global complex and proves the cohomology and commuting projection results. \Cref{sec:2d,sec:3d} present the two- and three-dimensional examples, respectively.   \Cref{sec:conclu} summarizes the conclusions. 

\section{Framework and main results}
\label{sec:main-results}
In this section, we present the main results regarding the discrete complexes. Throughout this paper, these complexes are based on finite elements. 

This section is organized as follows. First, in \Cref{sec:ts}, we introduce the trace structure. Finite element spaces and finite element complexes equipped with the trace structure are also presented in this section. In \Cref{sec:comp}, we briefly review differential complexes. In particular, we introduce the cohomology of the de Rham complex. In \Cref{sec:gen-cur}, we define generalized currents and give examples for tensor-valued complexes in two and three dimensions. In \Cref{sec:main-result}, we give the definition of the finite element complexes with trace structures (FECTS). We then state the cohomology and bounded commuting projection results.

We begin by introducing notations, conventions, and assumptions that are frequently used in this paper. 

Let $\mathcal T$ be a conforming simplicial mesh of a Lipschitz domain $\Omega \subset \mathbb R^n$. Denote by $\mathcal T_s$  the collection of all $s$-dimensional faces (subsimplices) of $\mathcal T$, and $\mathcal T_{\leq s}$  the collection of all faces of $\mathcal T$ with dimension not greater than $s$. Generally, we will use Greek letters $\sigma, \tau,\eta$ to denote a simplex of $\mathcal T$. In particular, denote by $\mathcal{V},\mathcal{E},\mathcal{F}$ and $\mathcal{K}$ the sets of all vertices, edges, faces, and tetrahedra, respectively.

For each $\sigma \in \mathcal T$, let $h_{\sigma}$  denote the diameter of $\sigma\in\mathcal{T}$, and let $h :=\max_{\sigma\in\mathcal{T}_n}h_{\sigma}$ denote the mesh size. 
We say $\tau \trianglelefteq \eta$ if $\tau$ is a subsimplex of $\eta$, and say $\tau \trianglelefteq_1 \eta$ if $\tau$ is a subsimplex of $\eta$ with codimension 1. For a simplex $\tau\in\mathcal{T}$, let $\st(\tau) := \cup\{\sigma \in \mathcal T_n : \tau \trianglelefteq \sigma\}$ and $\st^1(\tau) := \cup \{ \sigma \in \mathcal T_n : \tau \cap \sigma\neq \emptyset \} = \cup_{v \trianglelefteq \tau} \st(v)$ and $\st^2(\tau) = \st^1(\st^1(\tau))$. 

Throughout this paper, we denote the spaces of all $n\times n$ matrices,
symmetric $n\times n$ matrices, and traceless $n\times n$ matrices
by $\mathbb M^n$, $\mathbb S^n$, and $\mathbb T^n$, respectively.
We use the standard notation $L^2(\omega)$ and $H^m(\omega)$
for the Lebesgue and Sobolev spaces on a domain $\omega$,
and denote the $L^2$ inner product and norm by
$(\cdot,\cdot)_{\omega}$ and $\|\cdot\|_{L^2(\omega)}$, respectively. For any Hilbert space $\mathcal{X}$, denote the inner product by $(\cdot,\cdot)_{\mathcal{X}}$. Let $C^{\infty}(\omega)$ denote the space of smooth functions on $\omega$,
and let $C_0^{\infty}(\omega)$ be the subspace of smooth functions
with compact support in $\omega$.
For any nonnegative integer $m$, we define $H_0^m(\omega)$
as the closure of $C_0^{\infty}(\omega)$ in $H^m(\omega)$.
Let $P_k(\omega)$ denote the space of polynomials of total degree
at most $k$ on $\omega$, and let $P_k(\mathcal T)$ denote
the space of piecewise polynomials of total degree at most $k$.
For any finite-dimensional space $\mathbb X$, such as
$\mathbb R^n$, $\mathbb M^n$, $\mathbb S^n$, or $\mathbb T^n$,
we use $V(\omega)\otimes\mathbb X$ to denote the corresponding
$\mathbb X$-valued function space, where $V(\omega)$ is any of the scalar function spaces introduced above. For any subspace $V$ and finite-dimensional subspace $W$ of $L^2(\omega)$, denote by
\begin{equation*}
    V/W: =\{v\in V,(v,w)_{\omega}=0,\text{ for all }w\in W\}.
\end{equation*}

Some conventions about continuity are made here. We say a piecewise smooth function $u$ is in $C^1(\mathcal V)$ if its value and first derivatives are single-valued at all vertices. The notations $C^k(\mathcal V), C^k(\mathcal E)$ and $C^k(\mathcal F)$, for a given nonnegative integer $k$, are understood in a similar manner. 

\subsection{Trace structure: algebraic tools for additional smoothness}
\label{sec:ts}

For each simplex $\tau\in\mathcal{T}$, we define the finite element shape function space $A(\tau)$ as a finite-dimensional subspace of $C^{\infty}(\tau)\otimes \mathbb A^{\operatorname{dim}\tau}$, where $\mathbb A^{\operatorname{dim}\tau}$ could be $\mathbb R^n$ (vectors), $\mathbb M^n$ (matrices),  $\mathbb S^n$ (symmetric matrices), $\mathbb T^n$ (traceless matrices), etc. In this paper, all shape function spaces $A(\sigma),\sigma \in \mathcal T_{n}$, of discrete complexes are polynomial function spaces, and therefore subspaces of smooth functions. We now introduce the trace structure of the finite element space.

\begin{definition}[Trace Structure]
	For each $\sigma,\sigma' \in \mathcal T_n$ and $\tau \trianglelefteq \sigma,\sigma'$, assume that there exists a linear trace operator $$\tr_{\sigma \to \tau}: A(\sigma) \to A(\tau)$$ such that
    $\tr_{\sigma \to \tau} A(\sigma) = \tr_{\sigma' \to \tau}A(\sigma')$. In particular, set $\tr_{\sigma\rightarrow\sigma} = \operatorname{Id}$. For each $\tau 
    \trianglelefteq\sigma$ and $\eta \trianglelefteq \tau$, there also exists a linear trace operator $\tr_{\tau\to\eta}: A(\tau) \to A(\eta)$, such that for any $u\in A(\sigma)$ and $u'\in A(\sigma')$, 
\begin{equation}\label{eq:trace-commuting} \tr_{\sigma \to \eta} u = 0\Rightarrow \tr_{\tau \to \eta} \circ \tr_{\sigma \to \tau} u= 0,\end{equation}
\begin{equation}\label{eq:trace-commuting_2}
    \tr_{\sigma \to \eta} u = \tr_{\sigma' \to \eta} u'\Rightarrow \tr_{\tau \to \eta} \circ \tr_{\sigma \to \tau} u= \tr_{\tau \to \eta} \circ \tr_{\sigma' \to \tau} u'.
\end{equation}
Also, set $\tr_{\tau\to\tau} =\operatorname{Id}$. Then we call $(A, \operatorname{tr})$ a trace structure. 
\end{definition}
\begin{remark}
    In this paper, the shape function space $A(\eta)$ could be the zero-space $\{0\}$. This means that the finite element has no trace structure on $\eta$. In this case, set the trace operator $\tr_{\tau \to \eta} = 0$ for all $\eta \trianglelefteq \tau$.
\end{remark}

\begin{remark}
    One may expect the composition of traces is again a trace. This is not true in general for the trace structure, as the motivation for introducing the trace structure is exactly to break the composition law. Instead, we do require that the composition of the traces is ``included'' in the trace. One way to interpret \eqref{eq:trace-commuting} and \eqref{eq:trace-commuting_2} is as follows: Suppose that there exists a linear operator $\rho_{\tau,\eta}$ such that $\rho_{\tau,\eta} \circ \tr_{\sigma \to \eta} = \tr_{\tau \to \eta} \circ \tr_{\sigma \to \tau}$, then \eqref{eq:trace-commuting} and \eqref{eq:trace-commuting_2} hold. In practice, such $\rho_{\tau,\eta}$ can be easily found, as differential operators or orthogonal projections.
\end{remark}
Now we can define the global space.
\begin{definition}[Global Space] The global space with respect to a trace structure $(A, \operatorname{tr})$ is defined as follows:
\begin{equation*}
    \begin{split}
        \bm A = \{ u \in L^2(\Omega)\otimes \mathbb A^n: u|_{\sigma} \in A(\sigma), \tr_{\sigma \to \tau} u|_{\sigma} = \tr_{\sigma' \to \tau} u|_{\sigma'}, \\ \text{ for all } \tau \trianglelefteq \sigma, \sigma' \text{ and }\sigma, \sigma'\in\mathcal{T}_n \}.
    \end{split}
\end{equation*}
\end{definition}

The space $\bm A$ collects the discrete functions with single-valued traces, reflecting certain continuity. In the following text, we use the bold symbol $\bm A^{\bullet}$ to denote the global space. We use two examples to illustrate what the trace structure is, and how it reflects a certain finite element.

\begin{example}
Consider the following trace structure in two dimensions:
\begin{itemize}
\item[-] $\displaystyle A(f) := P_k(f), A(e) = P_k(e)$, and $ A(x) = \mathbb R\oplus \mathbb R^2$ for all $f\in\mathcal{F},e\in\mathcal{E}$ and $x\in\mathcal{V}$.
\item[-] $\tr_{f\to e}(u) = u|_e, \tr_{f\to x}(u) = u(x)\oplus \nabla u(x)$ for all edges $e$ on $f$ and vertices $x$ of $f$.
\item[-]  $\tr_{e\to x}(q) = q(x)\oplus \partial_{\bm t}q(x) \bm t$ for all vertices $x$ of $e$.
\end{itemize}
Here, $\bm t$ is the unit tangential vector of $e$. It is easy to see that for each $u\in \bm A$, $u\in C^1(\mathcal{V})\cap C^0(\mathcal{E})$, which means that the global space is the space of the Hermite finite element. See \Cref{fig:ts-hm} for an illustration.

\end{example}

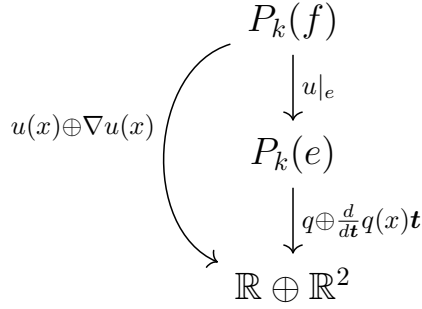
\begin{figure}[htbp]
\large 
\begin{equation*}
\begin{tikzcd}
P_k(f) \ar[d, "{u|_e}"] \ar[dd,swap, bend right=70, "{u(x)\oplus \nabla u(x)}"] \\ 
P_{k}(e)  \ar[d, "{ q(x)\oplus\partial_{\bm t}q(x) \bm t}"]\\ 
\mathbb R \oplus \mathbb R^2
\end{tikzcd}
\end{equation*}
\caption{The trace structure of the Hermite element.}
\label{fig:ts-hm}
\end{figure}

\begin{example}
Consider the following trace structure in two dimensions:

\begin{itemize}
\item[-] $\displaystyle A(f) := P_k(f), A(e) = P_k(e) \oplus P_{k-1}(e), A(x) = \mathbb R\oplus \mathbb R^2 \oplus \mathbb M^2$ for all $f\in\mathcal{F},e\in\mathcal{E}$ and $x\in\mathcal{V}$.
\item[-] $\displaystyle \tr_{f\to e}(u) = u|_e\oplus \partial_{\bm n}u|_e$, $\tr_{f \to x}(u)= u(x)\oplus\nabla u(x)\oplus \nabla^2u(x)$  for all edges $e$ on $f$ and vertices $x$ of $f$.
\item[-] $\footnotesize \tr_{e \to x}(u_1, u_2) = u_1(x)\oplus \big(u_2(x) \bm n+\partial_{\bm t}u_1(x) \bm t\big)\oplus \big(\partial_{\bm t}u_2(x) \bm n\bm t^T +  \partial_{\bm t}^2u_1(x) \bm t  \bm t^T\big)$ for all vertices $x$ of $e$.
\end{itemize}
Here, $\bm t$ and $\bm n$ are the unit tangential vector and normal vector of $e$, respectively. The global space is the space of the Argyris finite element.

\end{example}

To give a better picture relating the classical constructions of finite element spaces, let us introduce the definition of local bubble functions. Given $\sigma\in \mathcal{T}_n$, for any $\tau\trianglelefteq \sigma$, the local bubble space on $\tau\in\mathcal{T}$ consists of elements of the direct trace image whose traces vanish on every proper subsimplex of $\tau$.
\begin{definition}[Local Bubble Functions]
Define the bubble space $B(\tau)$ as  
$$ B(\tau) := \{ a \in\tr_{\sigma\rightarrow\tau}A(\sigma): \tr_{\tau\to \eta} a = 0 \text{ for all } \eta \trianglelefteq \tau  \text{ and }\eta \neq\tau\}.$$
\end{definition}
Next we present two examples of bubble spaces.
\begin{example}
The bubble spaces of the Hermite element are:
\begin{itemize}
\item[-] $B(v) = \mathbb R \oplus \mathbb R^2$, corresponding to the value and the gradient at vertex $v$ of the triangle $f$.
\item[-] $B(e) = \{p \in P_k(e): p(x) = 0,\tfrac{\partial p}{\partial \bm t} =0\text{ for all end points } x \text{ of }e\} = \lambda_e^2 P_{k-4}(e)$ for the edge $e$ of $f$.
\item[-] $B(f) = \{p\in P_k(f):p|_{\partial f}=0\}=\lambda_f P_{k-3}(f)$.
\end{itemize}\end{example}

\begin{example}
The bubble spaces of the Argyris element are:
\begin{itemize}
\item[-] $B(v) = \mathbb R \oplus \mathbb R^2 \oplus \mathbb S^{2}$ for the vertex $v$ of $f$.
\item[-] $B(e) = \lambda_e^3P_{k-6}(e) \oplus \lambda_e^2P_{k-5}(e)$ for the edge $e$ of $f$.
\item[-] $B(f) = \lambda_f^2 P_{k-6}(f).$
\end{itemize}

\end{example}
Here $\lambda_e$ and $\lambda_f$ denote the standard polynomial bubble functions on edge $e$ and face $f$, respectively.

In this paper, all the global finite element spaces will satisfy the following geometric decomposition:

\begin{definition}[Geometric Decomposition]
We say the trace structure satisfies the geometric decomposition if
\begin{equation*}
    \sum_{\tau \trianglelefteq \sigma}\operatorname{dim} B(\tau) = \operatorname{dim}A(\sigma),\quad \forall \sigma\in\mathcal{T}_n.
\end{equation*}
\end{definition}
The above definition leads to the following finite element construction.
\begin{proposition}[Finite element from the trace structure]\label{prop:fe}
Let $(A, \tr)$ be a trace structure satisfying geometric decomposition. For $\sigma\in\mathcal{T}_n$ with $\tau\trianglelefteq\sigma$, let $A(\sigma)$ denote the shape function space over $\sigma$, $B(\tau)$ the bubble space, and $B(\tau)^{\vee}$ its dual space, whose functionals extend to $A(\sigma)$ through the trace structure. Then, for $u\in A(\sigma)$, the local degrees of freedom are as follows: $L(\tr_{\sigma\to \tau} u)$ for any $L\in B(\tau)^{\vee}$. And the degrees of freedom are unisolvent for $A(\sigma)$. Moreover, for any $\sigma,\sigma'\in\mathcal T_n$ sharing a subsimplex $\tau$, if $u\in A(\sigma)$ and $u'\in A(\sigma')$ have the same degrees of freedom associated with all $\eta\trianglelefteq\tau$, then
$$
\tr_{\sigma\to\tau}u=\tr_{\sigma'\to\tau}u'.
$$
\end{proposition}
\begin{proof}
    Note that
 \begin{equation*}
    \sum_{\tau \trianglelefteq \sigma}\operatorname{dim}B(\tau)^{\vee} = \sum_{\tau \trianglelefteq \sigma}\operatorname{dim}B(\tau) = \operatorname{dim}A(\sigma),\quad \forall \sigma\in\mathcal{T}_n.
\end{equation*}   
It suffices to prove that for any  $u \in A(\sigma)$, if 
$$ L(\tr_{\sigma\to \tau} u)=0, \quad\forall \tau\trianglelefteq \sigma, L \in B(\tau)^{\vee},$$ then $u=0$. 

We first show that for any $ \tau \trianglelefteq \sigma$ and $\tau\neq \sigma$, it holds that $\tr_{\sigma\to \tau} u=0$. If $\tau\in\mathcal{T}_0$, it is easy to see that $\tr_{\sigma\to \tau} u=0$. Next we assume that for any $0\leq k\leq n-2,\tau\in\mathcal{T}_{k}$, we have $\tr_{\sigma\to \tau} u=0$. Then for any $\tau\in\mathcal{T}_{k+1}$ with $\tau \trianglelefteq \sigma$, note that for any $\eta\trianglelefteq\tau$ and $\eta\neq\tau$, it holds that
\begin{equation*}
      \tr_{\tau \to \eta} \circ \tr_{\sigma \to \tau}u=0.
\end{equation*}
Hence $\tr_{\sigma \to \tau}u\in B(\tau)$. From 
$ L(\tr_{\sigma\to \tau} u)=0, \quad\forall L \in B(\tau)^{\vee},$
we can get $\tr_{\sigma \to \tau}u =0$. The argument above shows that $u\in B(\sigma)$. Since
$L(u)=0$ for all $L\in B(\sigma)^{\vee}$,
it holds that $u=0$. The last assertion follows by induction on the dimension of
$\tau$, using \eqref{eq:trace-commuting_2} and the unisolvence
of the degrees of freedom on $B(\tau)$.
\end{proof}

\subsection{Recap on complexes: de Rham and beyond}\label{sec:comp}
This section briefly reviews complexes used in finite element methods, see  \cite{arnold2018finite} for more details. A (smooth) differential complex has the form
\begin{equation*}
\begin{tikzcd}
    0 \ar[r,""]& C^{\infty}(\Omega) \otimes \mathbb X^0 \ar[r,"d^0"]& C^{\infty}(\Omega) \otimes \mathbb X^1\ar[r,"d^1"]& \cdots \ar[r,"d^{n-1}"]& C^{\infty}(\Omega) \otimes \mathbb X^n \ar[r,""]&0.	
\end{tikzcd}
\end{equation*}
Each $\mathbb X^{i}$ is a linear space, such as $\mathbb R^n,\mathbb M^n,\mathbb S^n,\mathbb T^n$, etc. Each $d$ is a linear differential operator. In the examples considered in this paper, $d$ will be a first- or second- order operator. The complex requires $d \circ d = 0$, and therefore the cohomology can be defined:
\begin{equation*}
\mathcal H^k = \frac{\ker(d: C^{\infty}(\Omega) \otimes \mathbb X^{k} \to C^{\infty}(\Omega) \otimes \mathbb X^{k+1})}{\im(d: C^{\infty}(\Omega) \otimes \mathbb X^{k-1} \to C^{\infty}(\Omega) \otimes \mathbb X^k)}.
\end{equation*}

To construct $L^2$-bounded projections, we consider the following $L^2$ complex:
\begin{equation*}
\begin{tikzcd}
    0 \ar[r,""]& L^2(\Omega) \otimes \mathbb X^0 \ar[r,"d^0"]& L^2(\Omega) \otimes \mathbb X^1\ar[r,"d^1"]& \cdots \ar[r,"d^{n-1}"]& L^2(\Omega) \otimes \mathbb X^n \ar[r,""]&0.	
\end{tikzcd}
\end{equation*}
Note here all the differential operators are considered as unbounded operators. 

We also introduce the adjoint complex, which consists of the $L^2$ adjoint operators of the operators $d$:
\begin{equation*}
\begin{tikzcd}
    0 \ar[r,""]& L^2(\Omega) \otimes \mathbb X^n \ar[r,"(d^{n-1})^\star"]& L^2(\Omega) \otimes \mathbb X^{n-1}\ar[r,"(d^{n-2})^\star"]& \cdots \ar[r,"(d^0)^\star"]& L^2(\Omega) \otimes \mathbb X^0 \ar[r,""]&0.	
\end{tikzcd}
\end{equation*}
Specifically, the Hilbert-space adjoint operator $d^{\star}: L^2(\Omega)\otimes \mathbb X^{k} \to L^2(\Omega) \otimes \mathbb X^{k-1}$ is the unbounded operator satisfying
$$(d^{\star} a, b)_{\Omega} = (a, db)_{\Omega},\quad \forall~a\in \mathcal{D}(d^\star), b \in \mathcal{D}(d),$$
where $\mathcal{D}(d^\star),\mathcal{D}(d)$ are the domains of $d^\star$ and $d$, respectively. Next we introduce smooth complexes considered in this paper. The most important complex is the de Rham complex. For clarity, here we only consider the vector proxy version. 

In two dimensions, the de Rham complex reads as 
\begin{equation*}
    \begin{tikzcd}
        0 \ar[r,""] & C^{\infty}(\Omega)  \ar[r,"\grad"]& C^{\infty} (\Omega)\otimes \mathbb R^2 \ar[r,"\rot"] &C^{\infty}(\Omega) \ar[r,""]& 0,
    \end{tikzcd}
\end{equation*} 
while in three dimensions, it reads as 
\begin{equation*}
    \begin{tikzcd}
        0 \ar[r,""] & C^{\infty}(\Omega)  \ar[r,"\grad"]& C^{\infty} (\Omega)\otimes \mathbb R^3 \ar[r,"\curl"] &C^{\infty}(\Omega)\otimes \mathbb R^3 \ar[r,"\div"] &C^{\infty}(\Omega)\ar[r,""]& 0.
    \end{tikzcd}
\end{equation*} 
The cohomology determined by the de Rham complex is called de Rham cohomology, denoted as $\mathcal H_{dR}(\Omega)$. De Rham first proved that the de Rham cohomology is dual to the simplicial homology, where he proposed the idea of \emph{currents} \cite{deRham1984}. Integration over an oriented $k$-simplex defines a current acting on smooth $k$-forms. Based on the standard vector proxies, we define the following currents in three dimensions by assigning
\begin{enumerate}
	\item each vertex a linear functional $\delta_x \in (C^{\infty}(\Omega))' : u \mapsto u(x)$;
	\item each edge with the unit tangential vector $\bm t$ a linear functional $\delta_e \in (C^{\infty}(\Omega) \otimes \mathbb R^3)' : \bm u \mapsto \int_e \bm u \cdot \bm t$;
	\item each face with the unit normal vector $\bm n$ a linear functional $\delta_f \in (C^{\infty}(\Omega) \otimes \mathbb R^3)' : \bm u \mapsto \int_f \bm u \cdot \bm n$;
	\item each cell a linear functional $\delta_K \in (C^{\infty}(\Omega))' : u \mapsto \int_K  u$.
\end{enumerate}
By using the currents, de Rham established the duality between the simplicial homology and the de Rham cohomology, which will be also used in this paper. The key insight of his proof is based on the following Stokes formula: for any $\sigma \in \mathcal T_m$ and $w \in C^{\infty}(\Omega) \otimes \mathbb X^{m-1}$, it holds that 
$$\delta_{\sigma}(dw)= \delta_{\partial \sigma}(w) := \sum_{\tau \trianglelefteq_1 \sigma} \mathcal O(\tau, \sigma) \delta_{\tau}(w).$$
Here, the orientation is standardly defined as in \cite{hatcher2002algebraic}: for two oriented simplices $\tau$ and $\sigma = [v_0,\cdots, v_s]$, $\mathcal O(\tau,\sigma) = (-1)^j$ if $\tau =   [v_0,\cdots, \widehat{v_j},\cdots, v_s]$ ($\widehat{v_j}$ means deleting $v_j$), and $\mathcal O(\tau, \sigma) = 0$ otherwise. 

\subsection{Generalized currents}\label{sec:gen-cur}

There are also some other complexes, whose cohomology can be related to de Rham cohomology, see \cite{arnold2021complexes} for more details. Specifically, in this paper, the following complexes will be considered as examples. The definitions of the differential operators can be found in Sections \ref{sec:2d} and \ref{sec:3d}.
In two dimensions, we have the Hessian complex:
\begin{equation*}\begin{tikzcd} 
0 \ar[r,""] &  C^{\infty}(\Omega)  \ar[r,"\hess"] & C^{\infty}(\Omega)   \otimes \mathbb S^2\ar[r,"\rot "] & C^{\infty}(\Omega) \otimes \mathbb R^2 \ar[r] & 0,\end{tikzcd}
\end{equation*}
whose cohomology is isomorphic to $P_1(\Omega) \otimes \mathcal H_{dR}(\Omega)$, and the divdiv-complex
\begin{equation*}\begin{tikzcd} 0 \ar[r,""] &  C^{\infty}(\Omega)  \otimes \mathbb R^2 \ar[r,"\sym\curl"] & C^{\infty}(\Omega)   \otimes \mathbb S^2\ar[r,"\div\div "] & C^{\infty}(\Omega) \ar[r] & 0,\end{tikzcd}\end{equation*}
whose cohomology is isomorphic to $\mathcal{RT}^2 \otimes \mathcal H_{dR}(\Omega)$. Here $\mathcal{RT}^2 :=\{\bm a+b\bm x,\bm a\in\mathbb R^2,b\in\mathbb R\}$, whose dimension is 3.

In three dimensions, we have the Hessian complex
\begin{equation*}
\begin{tikzcd}[column sep=2em]
0 \ar[r]
& C^{\infty}(\Omega) \ar[r,"\hess"]
& C^{\infty}(\Omega)\otimes\mathbb S^3 \ar[r,"\curl"]
& C^{\infty}(\Omega)\otimes\mathbb T^3 \ar[r,"\div"]
& C^{\infty}(\Omega)\otimes\mathbb R^3 \ar[r]
& 0,
\end{tikzcd}
\end{equation*}
whose cohomology is isomorphic to ${P}_1(\Omega) \otimes \mathcal H_{dR}(\Omega)$, and the divdiv complex
\begin{equation*}\begin{tikzcd}[column sep=2em] 0 \ar[r,""] &  C^{\infty}(\Omega)  \otimes \mathbb R^3 \ar[r,"\dev\nabla"] & C^{\infty}(\Omega) \otimes \mathbb T^3\ar[r,"\sym\curl"] & C^{\infty}(\Omega) \otimes \mathbb S^3   \ar[r,"\div\div"] &  C^{\infty}(\Omega)   \ar[r] &0,\end{tikzcd}\end{equation*}
whose cohomology is isomorphic to $\mathcal{RT}^3 \otimes \mathcal H_{dR}(\Omega)$, and the elasticity complex
\begin{equation*}\begin{tikzcd}[column sep=1.8em] 0 \ar[r,""] &  C^{\infty}(\Omega) \otimes \mathbb R^3 \ar[r,"\sym\nabla"] & C^{\infty}(\Omega) \otimes \mathbb S^3  \ar[r,"\inc"] & C^{\infty}(\Omega) \otimes \mathbb S^3 \ar[r,"\div"] &   C^{\infty}(\Omega) \otimes \mathbb R^3 \ar[r] &0,\end{tikzcd}\end{equation*}
whose cohomology is isomorphic to $\mathcal{RM}^3 \otimes \mathcal H_{dR}(\Omega)$. Here $\mathcal{RT}^3:=\{\bm a+b\bm x,\bm a\in\mathbb R^3,b\in \mathbb R\}$, whose dimension is 4. And $\mathcal{RM}^3:=\{\bm a+\bm b\times \bm x,\bm a\in\mathbb R^3,\bm b\in\mathbb R^3\}$, whose dimension is 6.

\begin{definition}[Generalized Currents]
Let $\mathcal{H}$ be a finite-dimensional Hilbert space. Suppose that for the complex
	\begin{equation}
	\label{eq:diff-complex-l2} \begin{tikzcd} 0 \ar[r,""] &  C^{\infty}(\bar\Omega) \otimes \mathbb X^0 \ar[r,"d^0"] &   C^{\infty}(\bar\Omega)  \otimes \mathbb X^1\ar[r,"d^1"] & \cdots \ar[r,"d^{n-1}"] &   C^{\infty}(\bar\Omega)  \otimes \mathbb X^n \ar[r,""] &  0\end{tikzcd},\end{equation}
and each $\sigma \in \mathcal T_{m}$, there exists a linear surjective mapping $ \Upsilon_{\sigma} : C^{\infty}(\bar\Omega) \otimes \mathbb X^m \to \mathcal H$, such that the generalized Stokes formula holds for every $w \in C^{\infty}(\bar\Omega) \otimes \mathbb X^{m-1},1\leq m\leq n$:
\begin{equation*}
	\Upsilon_{\sigma}(dw) = \Upsilon_{\partial \sigma}(w):= \sum_{\tau \trianglelefteq_1 \sigma} \mathcal O(\tau, \sigma) \Upsilon_{\tau}(w).	
\end{equation*}
We then call $(\Upsilon_{\bullet
}, \mathcal H)$ a family of generalized currents for \eqref{eq:diff-complex-l2}. 
\end{definition}

Next, we introduce some examples of tensor complexes (e.g., Hessian, elasticity, divdiv). Note that the cohomology of each complex is isomorphic to $\mathcal{Z}\otimes \mathcal{H}_{dR}(\Omega)$, where $\mathcal{Z}$ is a finite-dimensional  subspace of $L^2(\Omega)$ with a normalized orthogonal basis $\{z_i\}_{1\leq i\leq \operatorname{dim}\mathcal{Z}}$. 

\subsubsection{Two-dimensional Results}

First, we introduce the generalized currents for the 2D Hessian complex. Note that now $\mathcal{Z} = P_1(\Omega)$. 
\begin{lemma}[Generalized currents for the 2D Hessian]
\label{lem:currents-2D-hessian}
Let $\psi_1,\psi_2  ,\psi_3$ be a basis of $\mathcal{RT}^2= \{\bm a + b \bm x : \bm a \in\mathbb R^2 , b \in \mathbb R \}$. Then set 
\begin{equation*}
    \Upsilon_x: C^{\infty}(\bar\Omega)  \to \mathcal Z, \quad\Upsilon_x(u)  = \sum_{i=1}^3\big((\psi_i \cdot \nabla u)(x) - \tfrac{1}{2}(\div \psi_i \cdot u)(x)\big)z_i,
\end{equation*}
\begin{equation*}
    \Upsilon_e: C^{\infty}(\bar\Omega) \otimes \mathbb S ^ 2\to \mathcal Z,\quad \Upsilon_e(\bm \sigma) =\sum_{i=1}^3 \big(\int_{e} (\bm \sigma \bm t) \cdot \psi_i\big)z_i,
\end{equation*}
and 
\begin{equation*}
    \Upsilon_f: C^{\infty}(\bar\Omega) \otimes \mathbb R^2  \to \mathcal Z,\quad \Upsilon_f(\bm p) = \sum_{i=1}^3\big(\int_{f} \bm p \cdot \psi_i\big)z_i,
\end{equation*}
for each $x\in\mathcal{V},e\in\mathcal{E}$ and $f\in\mathcal{F}$. Then, $(\Upsilon_{\bullet},\mathcal Z)$ is a family of generalized currents.
\end{lemma}
\begin{proof}
    By direct integrations by parts.
\end{proof}

Next, we introduce the generalized currents for the 2D divdiv complex. Here $\mathcal{Z} = \mathcal{RT}^2= \{\bm a + b \bm x : \bm a \in\mathbb R^2 , b \in \mathbb R \}$, then we can introduce the following generalized currents.
\begin{lemma}[Generalized Currents for the 2D divdiv]
\label{lem:currents-2D-divdiv}
	Let $\psi_1, \psi_2, \psi_3$ be a basis of $P_1(\Omega)$. Set 
    \begin{equation*}
        \Upsilon_{x}: C^{\infty}(\bar\Omega) \otimes \mathbb R^2 \to \mathcal Z,\quad \Upsilon_x(\bm u) = \sum_{i=1}^3\big( \tfrac{1}{2} (\div \bm u \psi_i)(x)-(\bm u \cdot \nabla \psi_i)(x)\big)z_i,
    \end{equation*}
    \begin{equation*}
        \Upsilon_{e}: C^{\infty}(\bar\Omega) \otimes \mathbb S^2 \to \mathcal Z, \quad \Upsilon_e(\bm \sigma) = \sum_{i=1}^3\big(\int_e (\div \bm \sigma\cdot \bm n)\psi_i-(\bm \sigma \bm n \cdot \nabla \psi_i)\big)z_i,
    \end{equation*}
    \begin{equation*}
        \Upsilon_{f}: C^{\infty}(\bar\Omega) \to \mathcal Z,\quad \Upsilon_f(q) = \sum_{i=1}^3\big(\int_f  q \psi_i\big)z_i,
    \end{equation*}
for each $x\in\mathcal{V},e\in\mathcal{E}$ and $f\in\mathcal{F}$. Then, $(\Upsilon_{\bullet},\mathcal Z)$ is a family of generalized currents.
\end{lemma}
\begin{proof}
    See \cite[Lemmas 4.2 and 4.4]{hu2025distributional}.
\end{proof}
\begin{remark}
The generalized currents were used to construct the distributional Hessian complex, as described in \cite{hu2025distributional}. The construction comes from the duality between the Hessian and divdiv complexes.
\end{remark}

\subsubsection{Three-dimensional Results}

Here, we introduce the generalized currents of tensor complexes in three dimensions. 
We first introduce the generalized currents for the Hessian complex in three dimensions.
\begin{lemma}[Generalized currents for the 3D Hessian]
\label{lem:currents-3d-hessian}
Note that $\mathcal Z = P_1(\Omega)$, let $\{\psi_i\}_{1\leq i\leq 4}$ be a basis of $\mathcal{RT}^3 = \{\bm a + b \bm x : \bm a \in\mathbb R^3 , b \in \mathbb R \}$. Set 
\begin{equation*}
    \Upsilon_x: C^{\infty}(\bar \Omega)  \to \mathcal Z, \quad \Upsilon_x(u)  =\sum_{i=1}^4\big( (\psi_i \cdot \nabla u)(x) - \tfrac{1}{3}(\div \psi_i \cdot u)(x)\big)z_i,
\end{equation*}
\begin{equation*}
    \Upsilon_e: C^{\infty}(\bar\Omega) \otimes \mathbb S ^3\to \mathcal Z, \quad\Upsilon_e(\bm \sigma) = \sum_{i=1}^4\big(\int_{e} (\bm \sigma \bm t) \cdot \psi_i\big)z_i,
\end{equation*}
\begin{equation*}
    \Upsilon_f: C^{\infty}(\bar\Omega) \otimes \mathbb T^3  \to \mathcal Z,\quad \Upsilon_f(\bm \tau) = \sum_{i=1}^4\big(\int_{f}(\bm \tau \bm n) \cdot \psi_i\big)z_i,
\end{equation*}
\begin{equation*}
    \Upsilon_T: C^{\infty}(\bar\Omega) \otimes \mathbb R^3  \to \mathcal Z, \quad\Upsilon_T(\bm q) = \sum_{i=1}^4\big(\int_{T} \bm q \cdot \psi_i\big)z_i,
\end{equation*}
for any $x\in\mathcal{V},e\in\mathcal{E},f\in\mathcal{F}$ and $T\in\mathcal{K}$. Then, $(\Upsilon_{\bullet},\mathcal Z)$ is a family of generalized currents.
\end{lemma}
\begin{proof}
    By direct integrations by parts.
\end{proof}
\begin{lemma}[Generalized currents for the 3D divdiv]
\label{lem:currents-3D-divdiv}
Note that $\mathcal Z = \mathcal{RT}^3=\{\bm a+b\bm x,\bm a\in\mathbb R^3,b\in \mathbb R\}$, and let $\{\psi_i\}_{1\leq i\leq 4}$ be a basis of $P_1(\Omega)$. Set
\begin{equation*}
    \Upsilon_x: C^{\infty}(\bar\Omega) \otimes \mathbb R^3  \to \mathcal Z, \quad\Upsilon_x(\bm u) =   \sum_{i=1}^4\big( \frac{1}{3} (\psi_i \cdot \div \bm u)(x)-(\nabla \psi_i  \cdot \bm u)(x) \big)z_i ,
\end{equation*}
\begin{equation*}
    \Upsilon_e: C^{\infty}(\bar\Omega) \otimes \mathbb T^3\to \mathcal Z, \quad\Upsilon_e(\bm \tau) =\sum_{i=1}^4\big(  \int_{e}  \frac{1}{2} (\div \bm \tau^T \cdot \bm t) \cdot \psi_i-(\bm \tau \bm t) \cdot \nabla \psi_i\big)z_i,
\end{equation*}
\begin{equation*}
    \Upsilon_f: C^{\infty}(\bar\Omega) \otimes \mathbb S^3  \to \mathcal Z, \quad\Upsilon_f(\bm \sigma) = \sum_{i=1}^4\big( \int_{f}  (\div \bm  \sigma \cdot\bm n) \cdot \psi_i-(\bm \sigma \bm n) \cdot \nabla \psi_i\big)z_i,
\end{equation*}
\begin{equation*}
    \Upsilon_T: C^{\infty}(\bar\Omega)   \to \mathcal Z,\quad \Upsilon_T( q) = \sum_{i=1}^4\big(\int_{T} q \cdot \psi_i\big)z_i,
\end{equation*}
for any $x\in\mathcal{V},e\in\mathcal{E},f\in\mathcal{F}$ and $T\in\mathcal{K}$. Then, $(\Upsilon_{\bullet},\mathcal Z)$ is a family of generalized currents.
\end{lemma}
\begin{proof}
    See \cite[Propositions 5.3--5.5]{hu2025distributional}. 
\end{proof}

\begin{lemma}[Generalized currents for the 3D Elasticity]
\label{lem:currents-3D-elasticity}
Note that $\mathcal Z = \mathcal{RM}^3=\{\bm a+\bm b\times \bm x,\bm a\in\mathbb R^3,\bm b\in\mathbb R^3\}$, and let $\{\psi_i\}_{1\leq i\leq 6}$ be a basis of $\mathcal{RM}^3 $. Set 
\begin{equation*}
    \Upsilon_x: C^{\infty}(\bar\Omega)\otimes \mathbb R^3  \to \mathcal Z,\quad \Upsilon_x(\bm u)  = \sum_{i=1}^6\big(\frac{1}{2} (\psi_i \cdot  \curl \bm u)(x) + \frac{1}{2} (\curl \psi_i  \cdot \bm  u)(x)\big)z_i,
\end{equation*}
\begin{equation*}
    \Upsilon_e: C^{\infty}(\bar\Omega) \otimes \mathbb S ^3\to \mathcal Z,\quad \Upsilon_e(\bm \sigma) =\sum_{i=1}^6\big( \int_{e} (\frac{1}{2}\bm \sigma \bm t) \cdot \curl \psi_i +  ((\curl \bm \sigma)^T  \bm t) \cdot \psi_i\big)z_i,
\end{equation*}
\begin{equation*}
    \Upsilon_f: C^{\infty}(\bar\Omega) \otimes \mathbb S^3  \to \mathcal Z,\quad \Upsilon_f(\bm \tau) = \sum_{i=1}^6\big(\int_{f}(\bm \tau \bm n) \cdot \psi_i\big)z_i,
\end{equation*}
\begin{equation*}
    \Upsilon_T: C^{\infty}(\bar\Omega) \otimes \mathbb R^3  \to \mathcal Z, \quad\Upsilon_T(\bm q) = \sum_{i=1}^6\big(\int_{T} \bm q \cdot \psi_i\big)z_i,
\end{equation*}
for any $x\in\mathcal{V},e\in\mathcal{E},f\in\mathcal{F}$ and $T\in\mathcal{K}$. Then $(\Upsilon_{\bullet},\mathcal Z)$ is a family of generalized currents.
\end{lemma}
\begin{proof}
    See \cite[Section 5.1]{christiansen2026extended}.
\end{proof}

\subsection{FECTS: cohomologies and commuting projections}
\label{sec:main-result}
This subsection is devoted to finer structures of the FECTS. We first define finite element complexes with trace structures.
\begin{definition}
A sequence of finite element spaces $\bm{A}^k$ with the trace structure $(A^k, \tr^k)$ for $0\leq k\leq n$ forms a conforming finite element complex with the trace structure (FECTS) of
\begin{equation*}
\begin{tikzcd}
    0 \ar[r,""]& L^2(\Omega) \otimes \mathbb X^0 \ar[r,"d^0"]& L^2(\Omega) \otimes \mathbb X^1\ar[r,"d^1"]& \cdots \ar[r,"d^{n-1}"]& L^2(\Omega) \otimes \mathbb X^n \ar[r,""]&0,	
\end{tikzcd}
\end{equation*}
if the following conditions hold:
\begin{itemize}
\item[(1)] $\bm A^k\subset \mathcal{D}(d^k)$ and for any $u\in \bm A^k$, we have that $d^ku\in \bm A^{k+1}$.
\item[(2)] For any simplices $\tau\in\mathcal{T}$ and $\sigma\in\mathcal{T}_n$ with $\tau\trianglelefteq\sigma$, there exists a differential operator $d^k_{\tau}: A^k(\tau) \to A^{k+1}(\tau)$ such that $d^{k+1}_{\tau}\circ d^k_{\tau} =0$ on $\tr_{\sigma\rightarrow\tau} A^k(\sigma)$. And there holds the following commuting property:
\begin{equation*}
    d_{\tau}^k\tr_{\sigma\rightarrow\tau}^k u= \tr_{\sigma\rightarrow\tau}^{k+1}d_{\sigma}^ku,\quad \forall u\in A^k(\sigma),0\leq k\leq n-1.
\end{equation*}
Moreover, we have that
\begin{equation*}
    d_{\sigma}^k = d^k|_{\sigma},\quad 0\leq k\leq n-1,\sigma\in\mathcal{T}_n.
\end{equation*}

\end{itemize}
\end{definition}
\begin{remark}
   For simplicity, we write $d$ for $d_{\tau}^k$ or $d^k$,
and $\tr_{\sigma\rightarrow\tau}$ for
$\tr_{\sigma\rightarrow\tau}^k$ when no confusion arises.
\end{remark}
\begin{remark}
    We require only that the differential operators $d$ commute with traces from top-dimensional simplices.
\end{remark}
We suppose that there exists a finite-dimensional Hilbert space $\mathcal{Z}$ on $\Omega$ with the $L^2$ inner product, such that for any contractible domain $\omega\subset\Omega$, the following complex is exact:
\begin{equation*}
\begin{tikzcd}
\mathcal Z(\omega) \ar[r,"\subset"] & C^{\infty}(\omega) \otimes \mathbb X^{0} \ar[r,"d^0"] & \cdots  \ar[r,"d^{n-1}"] & C^{\infty}(\omega) \otimes \mathbb X^n \ar[r,""] & 0,
\end{tikzcd}
\end{equation*}
where $\mathcal{Z}(\omega)$ denotes the restriction of $\mathcal{Z}$ to $\omega$. The exactness of the above-mentioned complexes (i.e., de Rham, Hessian, divdiv, elasticity) can be found in \cite{arnold2021complexes}. 

\begin{definition}[Compatibility]
    We say that the FECTS $(A, \tr, d)$ is compatible with respect to the generalized currents $\{ \Upsilon_{\bullet}, \mathcal Z\}$ if the following conditions hold: 

\begin{enumerate}
	\item The generalized currents can be localized via the trace structures. That is, for each simplex $\tau$, there exists $\tilde{\Upsilon}_{\tau}: A^{\dim \tau}(\tau) \to \mathcal Z$ such that
    \begin{equation*}
        \widetilde{\Upsilon}_{\tau}(\tr_{\sigma \to \tau} w) = \Upsilon_{\tau}(w) \text{ for any }w \in \ A^{\operatorname{dim}\tau}(\sigma),\sigma\in\mathcal{T}_n, \text{ and }\tau\trianglelefteq \sigma.
    \end{equation*}
Since each $A^{k}(\sigma)$ for any $\sigma\in\mathcal{T}_n,0\leq k\leq n$ is a polynomial function space, the right-hand side is well-defined.
        \item It holds that $\mathcal{Z}\subset \bm A^0$. And the following sequence of bubble spaces
		\begin{equation}\label{eq:bubble_complex}\begin{tikzcd}[column sep = 1em]
       0\ar[r]& B^0(\tau) \ar[r,"d"]  & \cdots \ar[r,"d"] &B^{\dim(\tau)-1}(\tau) \ar[r,"d"] &  B^{\dim \tau}(\tau) \cap\ker(\widetilde{\Upsilon}_{\tau})  \ar[r,"d"] &  \cdots \ar[r,"d"] & B^{n}(\tau) \ar[r] & 0\end{tikzcd}\end{equation} 
       forms an exact complex for any $\tau\in\mathcal{T}$. Moreover, we require that
       \begin{equation*}
           \operatorname{dim}(B^{\dim \tau}(\tau) \cap\ker(\widetilde{\Upsilon}_{\tau})) = \operatorname{dim}B^{\dim \tau}(\tau) - \operatorname{dim}\mathcal{Z}.
       \end{equation*}
\end{enumerate}
\end{definition}
\begin{remark}
Since the commuting properties of $d$ and $\tr$ at each level are not required, \eqref{eq:bubble_complex} does not necessarily form a complex, and this property needs to be verified separately.
\end{remark}
Intuitively, the localization property reflects some requirements of the trace structure in the sense that the generalized currents must be part of the trace. The exactness condition shows that when the part contributed by generalized currents is removed, the remaining parts have essentially no contributions to cohomology.  
\begin{theorem}[Cohomology] \label{thm:cohomology}
Suppose that the FECTS satisfies the geometric decomposition assumption, and is compatible with respect to the generalized currents $\{ \Upsilon_{\bullet}, \mathcal Z\}$. Then the cohomology of the global discrete complex
\begin{equation*}
\begin{tikzcd}
0 \ar[r,""] & \bm A^0 \ar[r,"d"] & \cdots  \ar[r,"d"] & \bm A^n \ar[r,""] & 0
\end{tikzcd}
\end{equation*}
is isomorphic to $\mathcal Z \otimes \mathcal H_{dR}(\Omega)$. 
\end{theorem}

 A direct result from Theorem \ref{thm:cohomology} is that if the domain $\Omega$ is contractible and the FECTS satisfies the geometric decomposition assumption, and is compatible with respect to the generalized currents $\{ \Upsilon_{\bullet}, \mathcal Z\}$, then the following discrete complex
\begin{equation*}
\begin{tikzcd}
\mathcal Z \ar[r,"\subset"] & \bm A^0 \ar[r,"d"] & \cdots  \ar[r,"d"] & \bm A^n \ar[r,""] & 0
\end{tikzcd}
\end{equation*}
is exact.

In what follows, we will construct projections $\pi^k$ from $L^2(\Omega) \otimes \mathbb X^k$ to $\bm A^k$, such that 
$d \pi^k w = \pi^{k+1} d w$ whenever $w \in L^2(\Omega) \otimes \mathbb X^k$ and $dw\in L^2(\Omega) \otimes \mathbb X^{k+1}$. Specifically, we will construct $\pi^k$ in the following form.
\begin{definition}
We say a function $\pi: L^2(\Omega)\otimes \mathbb X \to \bm A$ is of inner-product type if 
$$\pi(w) := \sum_{i = 1}^I ( w, \mathsf \Omega_i )_{\Omega} \mathsf \Phi_i.$$
Here $\mathsf \Omega_i \in L^2(\Omega)\otimes \mathbb X$ and $\mathsf \Phi_i \in \bm A$.  
Moreover, we say that $\pi$ is a local operator in $\mathcal T$ if the restriction of $\pi(w)$ to each $\sigma\in\mathcal{T}_n$ only depends on the restriction of $w$ to $\st^2(\sigma)$.
\end{definition}

\begin{theorem}[Commuting projections]\label{thm:operator}
Suppose that the assumptions of Theorem \ref{thm:cohomology} hold, and that each element patch $\st^1 \tau$ is contractible for $\tau \in \mathcal T$. 

Then, there exists a projection $\pi^k:L^2(\Omega)\otimes \mathbb X^k\xrightarrow{} \bm A^k$, such that each $\pi^k$ is of local inner-product type. Furthermore, for any $u\in L^2(\Omega)\otimes \mathbb X^k$ with $ du \in L^2(\Omega)\otimes \mathbb X^{k+1}$, it holds that $d\pi^ku=\pi^{k+1}du$.
\end{theorem}
Finally, we consider the boundedness of the interpolation operators. Denote by $\operatorname{ord}(d_k)$ the order of $d_k$, and define $M:= \max_k \operatorname{ord}(d_k)$. For all the BGG complexes discussed in this paper, the generalized Bogovskii operators exist, namely, the following results hold \cite{vcap2023bounded,costabel2010bogovskiui}.

\begin{proposition}\label{prop:poincare_inequ}
    Let $\omega$ be a bounded, contractible, Lipschitz domain. Then for any $u\in H_0^{M}(\omega)\otimes \mathbb X^k,0\leq k< n$ with $(d^{k-1})^\star u=0$, and also $u\perp \mathcal{Z}(\omega)$ if $k=0$, there exists $p\in H_0^{M+\operatorname{ord}(d_k)}(\omega)\otimes \mathbb X^{k+1}\subset H_0^{M}(\omega)\otimes \mathbb X^{k+1}$, such that $(d^k)^\star p =u$ and $\|p\|_{L^2(\omega)}\leq C(\omega)\operatorname{diam}(\omega)^{\operatorname{ord}(d^k)}\|u\|_{L^2(\omega)}$. Here the Poincar\'{e} constant $C(\omega)$ depends on $\omega$.
\end{proposition}

Moreover, we assume that the generalized currents are homogeneous with respect to scaling. 
\begin{definition}[Homogeneity]
    We say that the generalized currents $\{\Upsilon_{\bullet}\}$ are homogeneous if for any $0\leq k\leq n$, there exists a constant $\ell_k\in \mathbb R$ such that for any $\tau \in\mathcal{T}_k$ and scaling map $\hat{\bm x}=\varphi(\bm x):=(\bm x-x_0)/a$ with a fixed vertex $x_0$ of $\tau$ and $a>0$ which maps $\tau$ to $\hat{\tau} = \varphi(\tau)$, it holds that
    \begin{equation*}
        \Upsilon_{\tau}(\phi) = a^{\ell_k} \Upsilon_{\hat{\tau}}(\hat{\phi}), \quad\text{ for any }\phi\in C^{\infty}(\Omega)\otimes \mathbb X^k,
    \end{equation*}
    where $\hat{\phi} =(\varphi^{-1})^*\phi =\phi\circ \varphi^{-1}$.
\end{definition}

\begin{theorem}\label{thm:operatorL2}
  Suppose moreover that the triangulation is shape regular, the generalized currents $\{ \Upsilon_{\bullet}\}$ are homogeneous, each patch $\st^1(\tau)$ is contractible and the Poincar\'{e} constant $C(\st^1(\tau))$ in Proposition \ref{prop:poincare_inequ} is uniformly bounded for all $\tau\in\mathcal{T}$. Then the operator $\pi^k$ constructed above is locally $L^2$ bounded: there is a constant $C$, depending only on the regularity constant, such that 
$$\| \pi^k u\|_{L^2(\sigma)} \le C \| u\|_{L^2(\st^2(\sigma))}.$$
A direct consequence is that 
$$\| \pi^k u\|_{L^2(\Omega)} \le C' \| u\|_{L^2(\Omega)}.$$
\end{theorem}

\begin{remark}
We emphasize that in this paper, we do not require the exactness of the differential complex in $\Omega$. Instead, we only need local exactness on the element patch. Therefore, our results are also applicable to topologically nontrivial domains.
\end{remark}

\section{Cohomology and commuting projections}
\label{sec:proofs}

In this section, we present a direct sum decomposition of the conforming finite element complex. We then prove Theorems \ref{thm:cohomology} and \ref{thm:operator}. The proof of Theorem \ref{thm:operatorL2} follows from a scaling argument, and we omit the details here.

\subsection{Skeletal degrees of freedom, decomposition and cohomology}
In this subsection, we prove the cohomology statement in Theorem \ref{thm:cohomology}, provided that the FECTS is compatible with respect to the generalized currents. The goal of this section is to show a finer structure of the global finite element space and identify the cohomology. This provides a refinement of the cohomological result stated in the previous section.

We first introduce the harmonic inner product, which has been intensively studied and used in \cite{arnold2006finite,arnold2021local}. In the following, for any $\omega\subset\Omega$, a function supported in $\omega$ can be naturally extended by zero to $\Omega$.
\begin{definition}[Harmonic inner product]
Suppose that $X' \xrightarrow{d'} X \xrightarrow{d} X''$ is an exact sequence of Hilbert spaces. Then, the following bilinear form
$$\langle u, v \rangle_{X} = (\mathcal P_{d'(X')} u , \mathcal P_{d'(X')} v)_X + (du, dv)_{X''}$$
defines an inner product. Here $\mathcal P_{d'(X')}$ is the projection operator onto $d'(X')$.
\end{definition}

For each $\tau$, we define
\begin{equation*}
    \tilde{B}^k(\tau) = 
    \begin{cases}
        B^k(\tau)\cap\ker(\widetilde{\Upsilon}_{\tau}),& k = \operatorname{dim}\tau,\\
        B^k(\tau),&\text{otherwise}.
    \end{cases}
\end{equation*}
Then from \eqref{eq:bubble_complex}, it can be shown that
\begin{equation}\begin{tikzcd}[column sep = small]\label{eq:bubble_complex_1}
       0\ar[r]& \tilde{B}^0(\tau) \ar[r,"d"] &  \tilde{B}^1(\tau) \ar[r,"d"] & \cdots   \ar[r,"d"] & \tilde{B}^{n}(\tau) \ar[r] & 0\end{tikzcd}\end{equation} 
is an exact complex. Let $\mathcal N_{\tau,0}$ be the dimension of the bubble space $\tilde{B}^0(\tau)$. Moreover, we set $\mathcal{N}_{\tau,-1}=0$. Note that each bubble space $\tilde{B}^k(\tau)$ can be seen as a Hilbert space with the $L^2$-norm on $\tau$. We next construct an orthonormal basis of $\tilde{B}^k(\tau)$ with respect to the harmonic inner products.
\begin{lemma}[Orthogonal basis of bubble space]\label{lem:ortho_basis}
    For each $\tau\in\mathcal{T}$, there exists a normalized orthogonal basis $\{\phi_{\tau,0}^i\}_{1\leq i\leq \mathcal{N}_{\tau,0}}$ of the bubble space $\tilde{B}^0(\tau)$ such that
    \begin{equation*}
        \langle \phi_{\tau,0}^i, \phi_{\tau,0}^j \rangle_{\tilde{B}^0(\tau)} = (d \phi_{\tau,0}^i,d \phi_{\tau,0}^j)_{\tau} = \delta_{ij},\quad \forall 1\leq i,j\leq \mathcal{N}_{\tau,0}.
    \end{equation*}
    For each $1\leq k\leq n$, there exists a set $\{\phi_{\tau,k}^i\}_{1\leq i\leq \mathcal{N}_{\tau,k}}$ such that the functions in $\{d\phi_{\tau,k-1}^{i}\}_{1\leq i\leq \mathcal{N}_{\tau,k-1}}$ and $ \{\phi_{\tau,k}^i\}_{1\leq i\leq \mathcal{N}_{\tau,k}}$ form a basis of the bubble space $\tilde{B}^k(\tau)$, and satisfy
    \begin{align*}
        \langle d\phi_{\tau,k-1}^i, \phi_{\tau,k}^j \rangle_{\tilde{B}^k(\tau)} &= (d\phi_{\tau,k-1}^i, \phi_{\tau,k}^j)_{\tau} = 0,\quad \forall 1\leq i\leq \mathcal{N}_{\tau,k-1},1\leq j\leq \mathcal{N}_{\tau,k},\\
        \langle d\phi_{\tau,k-1}^i, d\phi_{\tau,k-1}^j \rangle_{\tilde{B}^k(\tau)} & = (d\phi_{\tau,k-1}^i, d\phi_{\tau,k-1}^j)_{\tau} = \delta_{ij},\quad \forall 1\leq i,j\leq \mathcal{N}_{\tau,k-1},\\
        \langle \phi_{\tau,k}^i, \phi_{\tau,k}^j \rangle_{\tilde{B}^k(\tau)} & = (d\phi_{\tau,k}^i, d\phi_{\tau,k}^j)_{\tau} = \delta_{ij},\quad \forall 1\leq i,j\leq \mathcal{N}_{\tau,k}.
    \end{align*}
\end{lemma}
\begin{proof}
    This follows by induction and the exactness of the bubble complex.
\end{proof}

\begin{lemma}[Local degrees of freedom]
    For each $\tau\in\mathcal{T}$ and $0\leq k\leq n$, a basis of the dual space $ B^k(\tau)^{\vee}$ can be given as
    \begin{equation*}
        \begin{cases}
            \{\langle \cdot, \phi \rangle_{\tilde{B}^k(\tau)}\},&\phi\in \{d\phi_{\tau,k-1}^{i}\}_i\cup \{\phi_{\tau,k}^i\}_i,\quad \text{if }k \neq \operatorname{dim}\tau,\\
            \{\langle \cdot, \phi\rangle_{\tilde{B}^k(\tau)}\}\cup \widetilde{\Upsilon}_{\tau}(\cdot),&\phi\in \{d\phi_{\tau,k-1}^{i}\}_i\cup \{\phi_{\tau,k}^i\}_i,\quad \text{if }k = \operatorname{dim}\tau.
        \end{cases}
    \end{equation*}
    By the geometric decomposition, each $w\in A^k(\sigma),\sigma\in\mathcal{T}_n$ can be uniquely determined by
    \begin{itemize}
        \item[-] $\langle \tr_{\sigma \to \tau}w , \phi \rangle_{\tilde{B}^k(\tau)}$ for any $\phi\in \{d\phi_{\tau,k-1}^{i}\}_i\cup \{\phi_{\tau,k}^i\}_i$ and $\tau \trianglelefteq \sigma$,
        \item[-] $(\Upsilon_{\tau}(w) ,z_i)_{\mathcal{Z}}$ for any $\tau \trianglelefteq \sigma$ with $\operatorname{dim}\tau =k$ and $1\leq i\leq \operatorname{dim}(\mathcal{Z})$.
    \end{itemize}
\end{lemma}
\begin{proof}This follows directly from the normalized orthogonal basis of bubble spaces and $\widetilde{\Upsilon}_{\tau}(\tr_{\sigma \to \tau} w) = \Upsilon_{\tau}(w)$.
\end{proof}

The degrees of freedom provide a direct sum decomposition of the global finite element space
$$
\bm A^k =  \bm S^k(\mathcal{T}) \oplus \bigoplus_{\tau \in \mathcal T} \mathbb B^k(\tau).$$
	
Here, the space $\mathbb B^k(\tau)$ consists of all functions $w\in \bm A^k$ such that
\begin{itemize}
    \item[-] $\langle \tr_{\sigma \to \eta}w , \phi \rangle_{\tilde{B}^k(\eta)}=0$  for any $\phi\in \{d\phi_{\eta,k-1}^{i}\}_i\cup \{\phi_{\eta,k}^i\}_i$ and $\eta \in \mathcal{T},\eta \neq \tau$;
    \item[-]  $\Upsilon_{\eta}(w) =0$ for any $\eta \in \mathcal{T}_k$.
\end{itemize}
And the space $ \bm S^k(\mathcal{T})$ consists of all functions $w\in \bm A^k$ such that
\begin{itemize}
    \item[-] $\langle \tr_{\sigma \to \eta}w , \phi \rangle_{\tilde{B}^k(\eta)}=0$  for any $\phi\in \{d\phi_{\eta,k-1}^{i}\}_i\cup \{\phi_{\eta,k}^i\}_i$ and $\eta \in \mathcal{T}$.
\end{itemize}
\begin{remark}
    Note that from the definition of $\bm A^k$, it holds that for any $w\in \bm A^k$, $\tr_{\sigma \to \tau} w|_{\sigma} = \tr_{\sigma' \to \tau} w|_{\sigma'}$ for all $\tau \trianglelefteq \sigma, \sigma'$ and $\sigma, \sigma'\in\mathcal{T}_n$. Therefore the notation $\tr_{\sigma \to \eta}w$ is unambiguous.
\end{remark}

\begin{lemma}[Closedness]
For any $\tau\in\mathcal{T},0\leq k\leq n-1$ and $w \in \mathbb B^k(\tau)$, it holds that $
    dw\in \mathbb B^{k+1}(\tau)$; 
 for any $v\in \bm S^k(\mathcal{T})$, it holds that $dv \in \bm S^{k+1}(\mathcal{T})$.
\end{lemma}
\begin{proof}
	For $w \in \mathbb B^k(\tau)$, we have $dw \in \bm A^{k+1}$ and
    \begin{equation*}
        \Upsilon_{\theta}(dw) = \Upsilon_{\partial \theta}(w) = \sum_{\eta \trianglelefteq_1 \theta} \mathcal O(\eta, \theta) \Upsilon_{\eta}(w) = 0,\quad \text{for all }\theta\in \mathcal{T}_{k+1}.
    \end{equation*}
    It suffices to prove that $\langle \tr_{\sigma \to \eta}dw , \phi \rangle_{\tilde{B}^{k+1}(\eta)}=0$ for any $\phi\in \{d\phi_{\eta,k}^{i}\}_i\cup \{\phi_{\eta,k+1}^i\}_i$ and $\eta \in \mathcal{T},\eta \neq \tau$. Note that
    \begin{equation*}
    \begin{split}
        \langle \tr_{\sigma \to \eta}dw , \phi \rangle_{\tilde{B}^{k+1}(\eta)} = \langle d\tr_{\sigma \to \eta}w , \phi \rangle_{\tilde{B}^{k+1}(\eta)} =  (d\tr_{\sigma \to \eta}w , \mathcal P_{d(\tilde{B}^k(\eta))}\phi)_{\eta} =0,\\ \quad \text{ for all } \phi\in \{d\phi_{\eta,k}^{i}\}_i\cup \{\phi_{\eta,k+1}^i\}_i,
    \end{split}
    \end{equation*}
    where we use the fact $d\circ \tr_{\sigma \to \eta} = \tr_{\sigma \to \eta}\circ d$. Hence we can derive that  $dw\in \mathbb B^{k+1}(\tau)$. 

    For any $v\in \bm S^k(\mathcal{T})$, it is obvious that $dv\in\bm A^{k+1}$, and from a similar argument, we can get that
    \begin{equation*}
        \langle \tr_{\sigma \to \eta}dv , \phi \rangle_{\tilde{B}^{k+1}(\eta)} =0, \quad \text{ for all } \phi\in \{d\phi_{\eta,k}^{i}\}_i\cup \{\phi_{\eta,k+1}^i\}_i,\eta\in \mathcal{T}.
    \end{equation*}
    Hence $dv\in \bm S^{k+1}(\mathcal{T})$, which completes the proof.
\end{proof}

From the above lemma, we can get the decomposition of the finite element complex.
\begin{proposition}[Direct sum decomposition of finite element complexes]\label{prop:directsum_complex}
The complex 
\begin{equation}
\begin{tikzcd}
0 \ar[r] & \bm A^0 \ar[r,"d"] & \bm A^1\ar[r,"d"] & \cdots \ar[r,"d"] & \bm A^n \ar[r] &  0
\end{tikzcd}
\end{equation}
is a direct sum of 
\begin{equation}
\label{eq:complex-sk}
\begin{tikzcd}
0 \ar[r] &	\bm S^0(\mathcal T) \ar[r,"d"] & \bm S^1(\mathcal T) \ar[r,"d"] & \cdots \ar[r,"d"] & \bm S^n(\mathcal T) \ar[r] & 0
\end{tikzcd}
\end{equation}
and 
\begin{equation}
\label{eq:complex-b}
\begin{tikzcd}
0 \ar[r] &	 \mathbb B^0(\tau) \ar[r,"d"] &  \mathbb B^1(\tau) \ar[r,"d"] & \cdots \ar[r,"d"] & \mathbb B^n(\tau) \ar[r] & 0
\end{tikzcd}
\end{equation}
for all $\tau \in \mathcal T$.
\end{proposition}

In the remaining part of this subsection, we will discuss the cohomology of the skeletal complex \eqref{eq:complex-sk} and the bubble complex \eqref{eq:complex-b}. First, we introduce an extension operator.

\begin{definition}[Extension of bubble functions]
For any $\tau\in \mathcal{T}$ and $0\leq k\leq n$, we define the extension operator $\mathsf E^k_{\tau}: \tilde B^k(\tau) \to \mathbb B^k(\tau) \subset \bm A^k$ as follows: For any $b \in \tilde B^k(\tau)$, set the degrees of freedom of $\mathsf E^k_{\tau}b$ as
\begin{itemize}
    \item[-] $\langle \tr_{\sigma \to \tau}\mathsf E^k_{\tau}b , \phi \rangle_{\tilde{B}^k(\tau)}=  \langle b , \phi \rangle_{\tilde{B}^k(\tau)} $ for any $\phi\in \{d\phi_{\tau,k-1}^{i}\}_i\cup \{\phi_{\tau,k}^i\}_i$;
    \item[-] $\langle \tr_{\sigma \to \eta}\mathsf E^k_{\tau}b , \phi \rangle_{\tilde{B}^k(\eta)}=0$  for any $\phi\in \{d\phi_{\eta,k-1}^{i}\}_i\cup \{\phi_{\eta,k}^i\}_i$ and $\eta \in \mathcal{T},\eta \neq \tau$;
    \item[-]  $\Upsilon_{\eta}(\mathsf E^k_{\tau}b) =0$ for any $\eta \in \mathcal{T}_k$.
\end{itemize}
\end{definition}

From the definition of $\mathsf E_{\tau}^{k}$, 
we can readily obtain that
\begin{equation*}
    \tr_{\sigma \to \eta}\mathsf E^k_{\tau}b =0,\quad \forall b\in \tilde{B}^k(\tau)\text{ and }\eta\trianglelefteq\tau,\eta\neq\tau.
\end{equation*}
This leads to 
\begin{equation*}
    \tr_{\tau \to \eta}\circ \tr_{\sigma \to \tau}\mathsf E^k_{\tau}b= 0,
\end{equation*}
and $\Upsilon_{\eta}(\mathsf E^k_{\tau}b) =0$ for any $\eta \in \mathcal{T}_k$, which implies that
\begin{equation*}
    \tr_{\sigma \to \tau}\mathsf E^k_{\tau}b\in\tilde{B}^k(\tau),  \quad \forall b\in \tilde{B}^k(\tau).
\end{equation*}
Hence, the definition of $\mathsf E_{\tau}^{k}$ gives
\begin{equation*}
    \tr_{\sigma \to \tau}\mathsf E^k_{\tau}b= b,  \quad \forall b\in \tilde{B}^k(\tau).
\end{equation*}
Then $\tr_{\sigma \to \tau}$ gives an isomorphism from $\mathbb B^{k}(\tau)$ to $\tilde B^k(\tau)$ with the inverse operator $\mathsf E_{\tau}^{k}$.

Then we have the following result, showing that the bubble complex $\mathbb B^{\bullet}$ has trivial cohomology. 
\begin{proposition}[Exactness of bubble complexes]
For any $\tau\in \mathcal{T}$ and $0\leq k\leq n$, the trace operator $\tr_{\sigma \to \tau}$ gives an isomorphism from $\mathbb B^{k}(\tau)$ to $\tilde B^k(\tau)$. The inverse operator is given by $\mathsf E_{\tau}^{k}$. Moreover, it holds that the differential operator $d$ commutes with $\mathsf E^k_{\tau}$. Then we get that the bubble sequence \eqref{eq:complex-b}
\begin{equation*}
\begin{tikzcd}
0 \ar[r] &	 \mathbb B^0(\tau) \ar[r,"d"] &  \mathbb B^1(\tau) \ar[r,"d"] & \cdots \ar[r,"d"] & \mathbb B^n(\tau) \ar[r] & 0
\end{tikzcd}
\end{equation*}
is exact.
\end{proposition}
\begin{proof}
Note that the commuting property is implied by the commuting property of $\tr$. 
Hence $\mathsf E^k_{\tau}$ induces an isomorphism from the cohomology of the bubble complex  \eqref{eq:bubble_complex_1} to the cohomology of the bubble sequence \eqref{eq:complex-b}. Since the bubble complex \eqref{eq:bubble_complex_1} is exact, the bubble complex \eqref{eq:complex-b} is exact.
\end{proof}

In what follows, we show that the skeletal complex gives the nontrivial cohomology.

\begin{definition}[Extension of the skeletal functions]
We define the extension operator for the skeletal complex: $\mathsf E_{\mathsf S}^k:C^k(\mathcal{T};\mathcal{Z}) \to \bm S^k(\mathcal{T})$ such that 
$$\Upsilon_{\eta}(\mathsf E_{\mathsf S}^k(z \otimes \tau)) = \delta_{\eta\tau} z,$$
where $C^k(\mathcal{T};\mathcal{Z}) $ is the space of simplicial $k$-cochains with coefficients in $\mathcal{Z}$. Clearly, $\mathsf E_{\mathsf S}^k$ is an isomorphism, and therefore invertible. The left inverse operator $\mathsf L^k_{\mathsf S}$ can be given as follows:
$$\mathsf L_{\mathsf S}^k(w) = \sum_{\tau \in \mathcal T_k}\Upsilon_{\tau}(  w)\otimes \tau,$$
which can be extended to $\bm A^k$.
\end{definition}

\begin{proposition}
The skeletal lifting operator $\mathsf E_{\mathsf S}^{\bullet}$ gives an isomorphism from the simplicial cochain complex with coefficients in $\mathcal Z$
\begin{equation*}
\begin{tikzcd}
0 \ar[r] &	 C^0(\mathcal{T};\mathcal{Z}) \ar[r,"\partial^{\vee}"] & C^1(\mathcal{T};\mathcal{Z}) \ar[r,"\partial^{\vee}"] & \cdots \ar[r,"\partial^{\vee}"] & C^n(\mathcal{T};\mathcal{Z}) \ar[r] & 0
\end{tikzcd}
\end{equation*}
to the skeletal complex, where $\partial^{\vee}$ denotes the coboundary operator.
\end{proposition}
\begin{proof}
It can be readily shown that $\mathsf L_{\mathsf S}^{k+1} \circ d \circ \mathsf E_{\mathsf S}^k = \partial^{\vee}$. Therefore, we conclude the result.
\end{proof}

Now we can conclude the result in this subsection.
\begin{theorem}
The cohomology of $\bm A^{\bullet}$ is isomorphic to $\mathcal Z \otimes \mathcal H_{dR}(\Omega)$. Moreover, the isomorphism is induced by $\mathsf L_{S}^{\bullet}.$
\end{theorem}

\subsection{Locally bounded interpolation operators}
In this subsection, we introduce the locally bounded interpolation operators in Theorem \ref{thm:operator}. The construction is largely inspired by \cite{arnold2021local}.

The degrees of freedom give the following decomposition for each $w \in \bm A^k$:
\begin{equation}\label{eq:fe_decompo}
\begin{split}
    w = \sum_{\tau\in\mathcal{T}_k}\sum_{i=1}^{\operatorname{dim}\mathcal{Z}}( \Upsilon_{\tau}(w), z_{i})_{\mathcal{Z}}\mathsf E_{\mathsf S}^k(z_i \otimes \tau) +  \sum_{\tau\in\mathcal{T}}\sum_{i =1}^{ \mathcal{N}_{\tau,k}}\langle \tr_{\sigma \to \tau}w , \phi_{\tau,k}^i \rangle_{\tilde{B}^k(\tau)}\mathsf E^k_{\tau}(\phi_{\tau,k}^i)\\+ \sum_{\tau\in\mathcal{T}}\sum_{i=1}^{ \mathcal{N}_{\tau,k-1}}\langle \tr_{\sigma \to \tau}w , d\phi_{\tau,k-1}^i \rangle_{\tilde{B}^k(\tau)}\mathsf E^k_{\tau}(d\phi_{\tau,k-1}^i).
\end{split}
\end{equation}
Here $\{z_i\}_i$ is a normalized orthogonal basis of $\mathcal{Z}$ and $(\cdot,\cdot)_{\mathcal{Z}}$ denotes the inner product of $\mathcal{Z}$. We seek an operator of the form
\begin{equation}\label{eq:inter:decomp}
\pi^k(w) = 	\pi^k_{\bm S}(w) + \pi^k_{\bm B}(w),
\end{equation}
such that 
\begin{equation}\label{eq:sket-pi}
\pi^k_{\bm S}(w) = \sum_{\tau \in \mathcal T_k} \sum_{i=1}^{\dim \mathcal{Z}} (w, \mathsf \Xi_{\tau}^i)_{\st^{1}(\tau)}  \mathsf E_{\mathsf S}^k(z_i \otimes \tau)  ,
\end{equation}
and 
\begin{equation}
\label{eq:bubble-pi}
\pi^k_{\bm B}(w) = \sum_{\tau\in\mathcal{T}}\sum_{i=1}^{ \mathcal{N}_{\tau,k}} (w, d^{\star}\mathsf \Omega_{\tau, k+1}^{i})_{\st (\tau)}  \mathsf E_{\tau}^k(\phi_{\tau,k}^i) + \sum_{\tau\in\mathcal{T}}\sum_{i=1}^{ \mathcal{N}_{\tau,k-1}} (w,  \mathsf \Omega_{\tau, k}^{i})_{\st(\tau)}  \mathsf E_{\tau}^k(d\phi_{\tau,k-1}^i),
\end{equation}
with some $\mathsf \Xi_{\tau}^i\in H_0^{M}(\st^1(\tau))\otimes\mathbb X^{\operatorname{dim}\tau}$ and $\Omega_{\tau,k}^i\in H_0^{2M}(\st(\tau))\otimes\mathbb X^k$. In what follows, we consider the properties of $\pi_{\bm S}$ and $\pi_{\bm B}$ separately. 
One useful result is that the above construction \eqref{eq:bubble-pi} satisfies the commuting property automatically.
\begin{proposition}
It holds that $\pi_{\bm B}^k(dw) = d\pi_{\bm B}^{k-1}(w)$ for $w \in L^2(\Omega) \otimes \mathbb X^{k-1}$ with $dw \in L^2(\Omega) \otimes \mathbb X^k$.	
\end{proposition}
\begin{proof}
By definition, the left-hand side is \begin{align*}
\pi^k_{\bm B}(dw) =  & \sum_{\tau\in\mathcal{T}}\sum_{i} (dw, d^{\star}\mathsf \Omega_{\tau, k+1}^i)_{\st (\tau)}  \mathsf E_{\tau}^k(\phi_{\tau,k}^i) + \sum_{\tau\in\mathcal{T}}\sum_{i} (dw,  \mathsf \Omega_{\tau, k}^{i})_{\st(\tau)}  \mathsf E_{\tau}^k(d\phi_{\tau,k-1}^i) \\
= & \sum_{\tau\in\mathcal{T}}\sum_{i} (w, d^{\star}d^{\star}\mathsf \Omega_{\tau, k+1}^i)_{\st (\tau)}  \mathsf E_{\tau}^k(\phi_{\tau,k}^i) + \sum_{\tau\in\mathcal{T}}\sum_{i} (w, d^{\star} \mathsf \Omega_{\tau, k}^{i})_{\st(\tau)}  \mathsf E_{\tau}^k(d\phi_{\tau,k-1}^i) \\ 
= & \sum_{\tau\in\mathcal{T}}\sum_{i} (w, d^{\star} \mathsf \Omega_{\tau, k}^{i})_{\st(\tau)}  \mathsf E_{\tau}^k(d\phi_{\tau,k-1}^i).
\end{align*}
The right-hand side is 
\begin{align*}
	d\pi^{k-1}_{\bm B}(w) = &  \sum_{\tau\in\mathcal{T}}\sum_{i} (w, d^{\star}\mathsf \Omega_{\tau, k}^i)_{\st (\tau)}  d\mathsf E_{\tau}^{k-1}(\phi_{\tau,k-1}^i) + \sum_{\tau\in\mathcal{T}}\sum_{i} (w,  \mathsf \Omega_{\tau, k-1}^{i})_{\st(\tau)}  d\mathsf E_{\tau}^{k-1}(d\phi_{\tau,k-2}^i)\\
	= & \sum_{\tau\in\mathcal{T}}\sum_{i} (w, d^{\star}\mathsf \Omega_{\tau, k}^i)_{\st (\tau)}  \mathsf E_{\tau}^{k}(d\phi_{\tau,k-1}^i) + \sum_{\tau\in\mathcal{T}}\sum_{i} (w,  \mathsf \Omega_{\tau, k-1}^{i})_{\st(\tau)}  \mathsf E_{\tau}^{k}(d \circ d\phi_{\tau,k-2}^i)  \\
	= & \sum_{\tau\in\mathcal{T}}\sum_{i} (w, d^{\star}\mathsf \Omega_{\tau, k}^i)_{\st (\tau)}  \mathsf E_{\tau}^{k}(d\phi_{\tau,k-1}^i),
\end{align*}
where we use the fact that $d$ commutes with $\mathsf E_{\tau}^{\bullet}$. This completes the proof.
\end{proof}

The above result shows that the commuting property is automatically valid for such a type of $\pi^k_{\bm B}$. We now choose the weights so that $\pi^k_{\bm B}$ is a projection. Therefore, we can choose $$(w, d^{\star}\mathsf \Omega_{\tau, k+1}^{i})_{\st (\tau)} =\langle \tr_{\sigma \to \tau}w , \phi_{\tau,k}^i \rangle_{\tilde{B}^k(\tau)}$$
and
$$ (w,  \mathsf \Omega_{\tau, k}^{i})_{\st(\tau)} =  \langle \tr_{\sigma \to \tau}w , d\phi_{\tau,k-1}^i \rangle_{\tilde{B}^k(\tau)}$$
for all $w \in \bm A^k(\st(\tau))$. We show that the latter equation implies the former one.

\begin{proposition}
If $(w, \mathsf \Omega_{\tau,k}^i)_{\st(\tau)} =\langle \tr_{\sigma \to \tau}w , d\phi_{\tau,k-1}^i \rangle_{\tilde{B}^k(\tau)}$ for all $w \in \bm A^k(\st(\tau))$, $1\leq k\leq n$, then $(w', d^{\star}\mathsf \Omega_{\tau, k}^{i})_{\st (\tau)} =\langle \tr_{\sigma \to \tau}w' , \phi_{\tau,k-1}^i \rangle_{\tilde{B}^{k-1}(\tau)}$ for all $w' \in \bm A^{k-1}(\st(\tau)).$
\end{proposition}
\begin{proof}
	Note that
    \begin{align*}
        (w', d^{\star}\mathsf \Omega_{\tau, k}^{i})_{\st (\tau)} &= (dw', \mathsf \Omega_{\tau, k}^{i})_{\st (\tau)}\\
        & = \langle \tr_{\sigma \to \tau}dw' , d\phi_{\tau,k-1}^i \rangle_{\tilde{B}^k(\tau)}\\
        & = (d\tr_{\sigma \to \tau}w' , d\phi_{\tau,k-1}^i)_{\tau}\\
        & = \langle \tr_{\sigma \to \tau}w' , \phi_{\tau,k-1}^i \rangle_{\tilde{B}^{k-1}(\tau)}.
    \end{align*}
	This proves the result.
\end{proof}
 
There are many ways to determine $\mathsf \Omega_{\tau, k}^{i}$. We describe one construction using the Riesz representation theorem. Recall that $M$ is the maximum order of the differential operators $d^{k}$. For each $\tau\in\mathcal{T}$ and $ 1\leq k\leq n$, we define $\bm {b}_{\tau}$ as the bubble function supported on $\st (\tau)$ such that
\begin{equation*}
    \bm {b}_{\tau}\in C^0(\st (\tau)),\quad \bm {b}_{\tau}\geq 0,\quad \bm {b}_{\tau}|_{\partial \sigma} =0,\text{ for all }\sigma\subset \st (\tau),\sigma\in \mathcal{T}_n.
\end{equation*}
Note that $(\cdot,(\bm{b}_{\tau})^{2M}\cdot)_{\st(\tau)}$ defines a weighted $L^2$ inner product on $\bm A^k(\st(\tau))$. Then from the Riesz representation theorem, for each $\phi_{\tau,k-1}^i , 1\leq i\leq \mathcal{N}_{\tau,k-1}$, there exists $\tilde\Omega_{\tau, k}^{i} \in \bm A^k(\st(\tau))$ such that
\begin{equation*}
    (w, (\bm{b}_{\tau})^{2M}\tilde  \Omega_{\tau,k}^i)_{\st(\tau)} =\langle \tr_{\sigma \to \tau}w , d\phi_{\tau,k-1}^i \rangle_{\tilde{B}^k(\tau)} \text{ for all }w \in \bm A^k(\st(\tau)).
\end{equation*}
Next we choose $\mathsf \Omega_{\tau, k}^{i} = (\bm{b}_{\tau})^{2M}\tilde  \Omega_{\tau,k}^i$ so that $ \mathsf \Omega_{\tau, k}^{i}\in H_0^{2M}(\st(\tau))\otimes \mathbb X^k$, since $\tilde{\Omega}_{\tau,k}^i\in \bm A^k(\st(\tau))$ is piecewise smooth.

Next, we consider the interpolation operator to the skeletal complex. We first consider a fixed element $z \in \mathcal Z$. 

\begin{proposition}
For any $\tau\in\mathcal{T}$, there exists $\mathsf \Xi_{\tau} \in H_0^{M}(\st^1(\tau))  \otimes \mathbb X^{\dim \tau}$ such that 
\begin{equation}
\label{eq:xi-1}
(w, \mathsf \Xi_{\tau})_{\st^1(\tau)} = (\Upsilon_{\tau}(w) , z)_{\mathcal{Z}},\quad \text{ for any }w \in \bm A^{\operatorname{dim}\tau}(\st^1(\tau))
\end{equation}
 and 
\begin{equation}\label{eq:xi-2} d^{\star} \mathsf \Xi_{\tau} = \mathsf \Xi_{\partial \tau} := \sum_{\eta \trianglelefteq_1 \tau} O(\eta, \tau) \mathsf \Xi_{\eta},\quad \text{ for any }\tau\in\mathcal{T}_{\geq 1}.
\end{equation}
Here $\mathsf \Xi_{\tau} = \mathsf \Xi_{\tau}[z]$ depends on the choice of $z$.
\end{proposition}
\begin{proof}
We prove by an inductive construction on $\dim \tau$ that \eqref{eq:xi-1} and \eqref{eq:xi-2} hold. The construction boils down to two steps. 

\paragraph{\emph{First step}} Let $\tau \in\mathcal{V}$ be a vertex of the mesh. Following the similar argument as above, note that $(\cdot,(\bm{b}_{\tau}^1)^{M}\cdot)_{\st^1(\tau)}$ defines an $L^2$ inner product of $\bm A^0(\st^1(\tau))$, where $\bm{b}_{\tau}^1$ is the bubble function on $\st^1(\tau)$ defined analogously to $\bm b_{\tau}$. Then from the Riesz representation theorem, there exists $\tilde{\mathsf{\Xi}}_{\tau} \in \bm A^0(\st^1(\tau))$ such that
\begin{equation*}
    (w, (\bm{b}_{\tau}^1)^{M}\tilde{\mathsf{  \Xi}}_{\tau})_{\st^1(\tau)} =(\Upsilon_{\tau}(w) , z)_{\mathcal{Z}},\quad \text{ for any }w \in \bm A^0(\st^1(\tau)).
\end{equation*}
Then we can choose $\mathsf \Xi_{\tau} = (\bm{b}_{\tau}^1)^{M}\tilde{\mathsf{\Xi}}_{\tau}$ so that $ \mathsf \Xi_{\tau}\in H_0^{M}(\st^1(\tau))\otimes \mathbb X^0$, since $\tilde{\mathsf{\Xi}}_{\tau}$ is piecewise smooth.

\paragraph{\emph{Second step}} Assume that for any $\tau\in \mathcal{T}$ with $0\leq \operatorname{dim}\tau\leq k-1 \leq n-1$, \eqref{eq:xi-1} and \eqref{eq:xi-2} hold. Then for any $\tau\in \mathcal{T}$ with $\operatorname{dim}\tau = k$, consider $\mathsf \Xi_{\partial \tau} = \sum_{\eta \trianglelefteq_1 \tau} O(\eta, \tau) \mathsf \Xi_{\eta}$. Note that $\operatorname{supp}\mathsf\Xi_{\eta}\subset\st^1(\eta)\subset\st^1(\tau)$, then we get that $\mathsf \Xi_{\partial \tau} \in H_0^{M}(\st^1(\tau))\otimes\mathbb X^{k-1}$. If $k\geq 2$, by induction we have
\begin{equation*}
    d^\star\mathsf \Xi_{\partial \tau} = \sum_{\eta \trianglelefteq_1 \tau} O(\eta, \tau) \mathsf \Xi_{\partial\eta}=0.
\end{equation*}
Otherwise, if $k=1$, for any $z^0\in \mathcal{Z}$, we have that
\begin{align*}
    (z^0,\mathsf \Xi_{\partial \tau})_{\st^1(\tau)} &= \sum_{\eta \trianglelefteq_1 \tau} O(\eta, \tau) (z^0,\mathsf \Xi_{\eta})_{\st^1(\tau)}\\
    & = \sum_{\eta \trianglelefteq_1 \tau} O(\eta, \tau) (\Upsilon_{\eta}(z^0) , z)_{\mathcal{Z}}\\
    & = (\Upsilon_{\partial \tau}(z^0), z)_{\mathcal{Z}} = (\Upsilon_{\tau}(dz^0), z)_{\mathcal{Z}}=0,
\end{align*}
which implies that $\mathsf \Xi_{\partial \tau}\in H_0^{M}(\st^1(\tau))\otimes\mathbb X^{0}/\mathcal{Z}(\st^1(\tau))$. Since the element patch $\st^1(\tau)$ satisfies the assumptions of Proposition \ref{prop:poincare_inequ}, there exists $\widetilde{\mathsf \Xi}_{\tau} \in H_0^{M }(\st^1(\tau))  \otimes \mathbb X^{k} $ such that $d^\star\widetilde{\mathsf \Xi}_{\tau} = \mathsf \Xi_{\partial \tau}$ and
\begin{equation*}
    \|\widetilde{\mathsf \Xi}_{\tau}\|_{L^2(\st^1(\tau))}\leq C h^{l}\|\mathsf \Xi_{\partial \tau}\|_{L^2(\st^1(\tau))},
\end{equation*}
where $l$ is the order of the differential operator $d^{k-1}$.

Next, we consider the modification of the following form: $\mathsf \Xi_{\tau} = \widetilde{\mathsf \Xi}_{\tau}  + d^{\star} \mathsf \Theta_{\tau}$. Clearly, in this case \eqref{eq:xi-2} always holds. Note that by Theorem \ref{thm:cohomology}, the discrete complex on $\st^1(\tau)$
\begin{equation*}
\begin{tikzcd}
\mathcal Z \ar[r,"\subset"] & \bm A^0(\st^1(\tau)) \ar[r,"d"] & \cdots  \ar[r,"d"] & \bm A^n(\st^1(\tau)) \ar[r,""] & 0
\end{tikzcd}
\end{equation*}
is exact. Denote by $\mathcal{N}(d)$ the kernel space of $d$, then there exists $\tilde{\mathsf \Theta}_{\tau} \in \bm A^{k}(\st^1 (\tau))/\mathcal{N}(d)$ such that
$$(du, (\bm b_{\tau}^1)^{2M}d\tilde{\mathsf \Theta}_{\tau})_{\st^1(\tau)} = (\Upsilon_{\tau}(u),z)_{\mathcal{Z}} - (u,\widetilde{\mathsf \Xi}_{\tau} )_{\st^1(\tau)},\quad \text{ for all }u \in \bm A^{k}(\st^1 (\tau))/\mathcal{N}(d).$$
Now choose $\mathsf\Theta_{\tau} = (\bm b_{\tau}^1)^{2M}d\tilde{\mathsf \Theta}_{\tau}$, we then prove the identity \eqref{eq:xi-1} for $w \in  \bm A^{k}(\st^1(\tau))$. If $w\in \bm A^{k}(\st^1 (\tau))/\mathcal{N}(d)$, it follows from the definition of $\mathsf \Theta_{\tau}$. Otherwise, if $dw=0$, from the exactness of the discrete complex, there exists $u'\in \bm A^{k-1}(\st^1 (\tau))$ such that $w = du'$, then it holds that 
\begin{align*}
    (du', \widetilde{\mathsf \Xi}_{\tau}  + d^{\star} \mathsf \Theta_{\tau})_{\st^1(\tau)} = (u', d^\star\widetilde{\mathsf \Xi}_{\tau} )_{\st^1(\tau)} = (u', \mathsf \Xi_{\partial\tau})_{\st^1(\tau)} = (\Upsilon_{\partial\tau}(u') ,z)_{\mathcal{Z}} = (\Upsilon_{\tau}(du') ,z)_{\mathcal{Z}},
\end{align*}
which completes the proof.
\end{proof}
Set $\mathsf \Xi_{\tau}^i = \mathsf \Xi_{\tau}[z_i]$, then the construction \eqref{eq:sket-pi} satisfies the commuting property.
\begin{proposition}
It holds that $\pi_{\bm S}^k(dw) = d\pi_{\bm S}^{k-1}(w)$ for $w \in L^2(\Omega) \otimes \mathbb X^{k-1}$ with $dw \in L^2(\Omega) \otimes \mathbb X^k$.	
\end{proposition}
\begin{proof}
    By definition, it holds that
    \begin{align*}
        \pi_{\bm S}^k(dw) & = \sum_{\tau \in \mathcal T_k} \sum_{i=1}^{\dim \mathcal{Z}} (dw, \mathsf \Xi_{\tau}^i)_{\st^{1}(\tau)}  \mathsf E_{\mathsf S}^k(z_i \otimes \tau)\\
        & = \sum_{\tau \in \mathcal T_k} \sum_{i=1}^{\dim \mathcal{Z}} (w, d^\star\mathsf \Xi_{\tau}^i)_{\st^{1}(\tau)}  \mathsf E_{\mathsf S}^k(z_i \otimes \tau)\\
        & = \sum_{i=1}^{\dim \mathcal{Z}}\sum_{\tau \in \mathcal T_k}\sum_{\eta \trianglelefteq_1 \tau} O(\eta, \tau)(w, \mathsf \Xi_{\eta}^i)_{\st^{1}(\eta)}\mathsf E_{\mathsf S}^k(z_i \otimes \tau)\\
        & = \sum_{i=1}^{\dim \mathcal{Z}}\sum_{\eta \in \mathcal T_{k-1}}(w, \mathsf \Xi_{\eta}^i)_{\st^{1}(\eta)}\sum_{\eta \trianglelefteq_1 \tau}  O(\eta, \tau)\mathsf E_{\mathsf S}^k(z_i \otimes \tau)\\
        & = \sum_{i=1}^{\dim \mathcal{Z}}\sum_{\eta \in \mathcal T_{k-1}}(w, \mathsf \Xi_{\eta}^i)_{\st^{1}(\eta)}d \mathsf E_{\mathsf S}^{k-1}(z_i \otimes \eta) = d\pi_{\bm S}^{k-1}(w),
    \end{align*}
    where the last equation follows from the property of $\Upsilon_{\bullet}$ and the definition of $\mathsf E_{\mathsf S}^{k-1}(z_i \otimes \eta)$, which completes the proof.
\end{proof}

Now we can conclude the result in this subsection.
\begin{theorem}
Suppose that the conditions in Theorem \ref{thm:operator} hold. Then the operators $\pi^k$ constructed in \eqref{eq:inter:decomp} are commuting projections. Furthermore, if the conditions in Theorem \ref{thm:operatorL2} hold true, then these projections $\pi^k$ are locally $L^2$ bounded, namely, there is a constant $C$, depending only on the regularity constant, and 
$$\| \pi^k u\|_{L^2(\sigma)} \le C \| u\|_{L^2(\st^2(\sigma))},\quad \text{ for all }u\in L^2(\Omega)\otimes \mathbb X^k,\sigma\in\mathcal{T}_n.$$
A direct consequence is that 
$$\| \pi^k u\|_{L^2(\Omega)} \le C' \| u\|_{L^2(\Omega)},\quad \text{ for all }u\in L^2(\Omega)\otimes\mathbb X^k.$$
\end{theorem}
\begin{proof}
    The proof is similar to that of \cite[Theorem~3.1]{arnold2021local}, which comes from scaling argument and compactness argument and Lemma \ref{lem:L2estimate:basis}.
\end{proof}

\section{Finite element complexes with trace structures: 2D case}
\label{sec:2d}
This section explores finite element complexes with trace structures in two-dimensional spaces. The resulting complexes have already been presented in the existing literature. Consequently, the primary emphasis of this section (and the subsequent section as well) is to point out that these constructions indeed possess finer structures.

For each complex, $\mathcal{Z}$ is a finite-dimensional Hilbert space and therefore admits an orthonormal basis $\{z_i\}_{1\leq i\leq\operatorname{dim}\mathcal{Z}}$.

\subsection{Notations}
This subsection collects the notations frequently used in the sequel. Suppose that $f$ lies in the $x-y$ plane. In two dimensions, for a given edge $e$, let $\bm t=(t_1,t_2)^T$, $\bm n=(t_2,-t_1)^T$ be its unit tangential and normal vectors, respectively. The corresponding directional derivatives are denoted by
$\partial_{\bm t}$ and
$\partial_{\bm n}$. For a scalar function $u$ and vector function $\bm q =(q_1,q_2)^T$, define the differential operators by
\begin{align*}
    \operatorname{grad} u =\nabla u=(\partial_{x}u,\partial_{y}u)^T,&\quad \operatorname{curl} u = (\partial_{y}u,-\partial_{x}u)^T, \\
    \operatorname{div} \bm q = \partial_{x}q_1 + \partial_{y}q_2,&\quad \operatorname{rot} \bm q = \partial_{x}q_2 - \partial_{y}q_1.
\end{align*}
For vector and tensor fields, these differential operators act row-wise. Define the Hessian operator by
\begin{equation*}
    \hess u =\nabla^2u =\begin{pmatrix}
        \partial_{x}^2u & \partial_{x}\partial_yu\\
        \partial_{y}\partial_xu &\partial_{y}^2u
    \end{pmatrix} .  
\end{equation*}
Furthermore, given a matrix $\bm M\in\mathbb M^{ n}$, define
\begin{equation*}
    \operatorname{sym}\bm M =\tfrac{1}{2}(\bm M +\bm M^T),\quad \operatorname{dev}\bm M = \bm M - \tfrac{1}{n}\operatorname{Tr}(\bm M)I_{n\times n},
\end{equation*}
where $\operatorname{Tr}(\bm M)$ denotes the trace of the matrix $\bm M$.

Next, we introduce some bubble polynomial function spaces on edges and faces. For any nonnegative integer $l$, define the bubble space $\mathring P_k^{(l)}(e)$ on each edge $e\in\mathcal{E}$ by
\begin{equation*}
    \mathring P_k^{(l)}(e):=\{ u\in  P_k(e): \partial_{\bm t}^{m}u\text{ vanishes at both end points of }e\text{ with } 0\leq m\leq l\}.
\end{equation*}
Similarly, for any nonnegative integers $m,l$, define the bubble space $\mathring P_k^{(m,l)}(f)$ on each face $f\in\mathcal{F}$ by
\begin{align*}
    \mathring P_k^{(m,l)}(f)&:= \{u\in P_k(f): D^{\alpha}u\text{ vanishes at all vertices of }f\text{ with }|\alpha|\leq m,\\
    & \ \ \ \ \ \ \ \ \ \ D^{\beta}u\text{ vanishes on all edges of }f\text{ with }|\beta|\leq l\}.
\end{align*}
In particular, define the bubble space $\mathring P_k^{(m,-1)}(f)$ on each face $f\in\mathcal{F}$ by
\begin{align*}
    \mathring P_k^{(m,-1)}(f): = \{u\in P_k(f): D^{\alpha}u\text{ vanishes at all vertices of }f\text{ with }|\alpha|\leq m\}.
\end{align*}

We next introduce the vertex jets, a concept that simplifies the expressions in the trace complexes (cf. \cite{saunders1989,christiansen2018nodal}). 
For each nonnegative integer $m$ and $x\in\mathcal{V}$, we define $\mathbb H^m(x)$ to be the space of constant tensors of order $m$ in $n$ dimensions. It can be observed that $\mathbb H^0(x) \cong \mathbb R,$ $\mathbb H^1(x) \cong \mathbb R^n$ and $\mathbb H^2(x) \cong \mathbb R^{n\times n}$. 
Then define the $m$-jet as 
\begin{equation*}
    \mathbb J^m(x) = \mathbb H^0(x) \oplus \mathbb H^1(x)\oplus \cdots \oplus\mathbb H^m(x).
\end{equation*}
For the non-scalar case, we use the tensor product-type space to denote the corresponding jets. For example, $\mathbb J^m(x) \otimes \mathbb R^n$ denotes the space of $m$-jets of vector-valued functions. Then we can define the differential operators on jets. Here we present two examples.
\begin{example}
    Consider a $1$-jet with respect to a two-dimensional vector function $\bm v$ at $x$ as 
\begin{equation*}
    \begin{pmatrix}
        v^1\\
        v^2\\
    \end{pmatrix}
    \oplus
    \begin{pmatrix}
        v^1_1& v^1_2\\
        v^2_1&v_2^2
    \end{pmatrix}
    \in\mathbb{J}^1(x)\otimes\mathbb R^2,
\end{equation*}
where $v^1$ and $v^2$ are two components of $\bm v$ and $v^i_j$ is the derivative of the $i$-th component with respect to the $j$-th coordinate. Its divergence is
\begin{equation*}
    \div \big(\begin{pmatrix}
        v^1\\
        v^2\\
    \end{pmatrix}
    \oplus
    \begin{pmatrix}
        v^1_1& v^1_2\\
        v^2_1&v_2^2
    \end{pmatrix}\big) = v^1_1+v^2_2\in \mathbb J^0(x).
    \end{equation*}
\end{example}
\begin{example}
    Consider a $2$-jet with respect to a scalar function $v$ at $x$ in 2D:
\begin{equation*}
 v\oplus
    \begin{pmatrix}
        v_1\\
        v_2\\
    \end{pmatrix}
    \oplus
    \begin{pmatrix}
        v_{11}& v_{12}\\
        v_{21}&v_{22}
    \end{pmatrix}
    \in\mathbb{J}^2(x),
\end{equation*}
where $v_{i}$ is the first derivative of $v$ and $v_{ij}$ is the second derivative of $v$. Its gradient is
\begin{equation*}
    \grad \big(v\oplus
    \begin{pmatrix}
        v_1\\
        v_2\\
    \end{pmatrix}
    \oplus
    \begin{pmatrix}
        v_{11}& v_{12}\\
        v_{21}&v_{22}
    \end{pmatrix}\big) = 
    \begin{pmatrix}
        v_1\\
        v_2\\
    \end{pmatrix}
    \oplus
    \begin{pmatrix}
        v_{11}& v_{12}\\
        v_{21}&v_{22}
    \end{pmatrix}\in \mathbb J^1(x) \otimes\mathbb R^2.
    \end{equation*}
\end{example}
Other differential operators ($\curl,\hess,\sym\curl,\div\div$, etc.) can be represented in a similar way.

In particular, denote the space of symmetric $m$-tensors by $\tilde{\mathbb S}^m(x)$. The space of symmetric tensors of order $m$ is the range of the $m$-th derivative of a smooth function at a point and is naturally isomorphic to the space of homogeneous polynomials of degree $m$. Therefore, we can define the bubble space at vertex $x\in\mathcal{V}$
\begin{equation*}
    \mathbb J_B^{m}(x) = 
    \tilde{\mathbb S}^0(x) \oplus \cdots \oplus \tilde{\mathbb S}^m(x).
\end{equation*}
Then there exists a natural isomorphism between  $\mathbb J_B^m(x)$ and the polynomial space $P_m(\Omega)$.

We first consider the smooth de Rham complex in 2D:
	\begin{equation}\label{eq:complex-dR-2}\begin{tikzcd} 0 \ar[r,""] &  C^{\infty}(\Omega) \otimes \mathbb R \ar[r,"\grad"] & C^{\infty}(\Omega) \otimes \mathbb R^2 \ar[r,"\rot"] &   C^{\infty}(\Omega) \otimes \mathbb R \ar[r] &0.\end{tikzcd}\end{equation}

Note that now $\mathcal{Z}=\mathbb R$. Then we can choose $\Upsilon_{\bullet}$ as follows: $\Upsilon_{f} :u \mapsto \int_f u$, $\Upsilon_e :\bm v \mapsto \int_e \bm v \cdot \bm t$ and  $\Upsilon_x :u \mapsto u(x)$. Therefore, it holds that $(\Upsilon_{\bullet}, \mathbb R) $ is a family of generalized currents of \eqref{eq:complex-dR-2}.
\subsection{Hermite--Stenberg de Rham complex}
Consider the following FECTS with $k\geq 3$.
\begin{equation}
\label{eq:fes-HmdR}
\begin{tikzcd}
    P_k(f) \arrow[r, "\grad"] \arrow[d, "u|_e"]  \ar[dd, bend right=30, "", swap] & P_{k-1}(f) \otimes \mathbb R^2 \arrow[r, "\rot"] \arrow[d, "\bm v|_e \cdot  \bm t"] \ar[dd, bend left=60, "{} "] & P_{k-2}(f) \\
    P_k(e) \arrow[r, "\partial_{\bm t}"] \arrow[d, "{u(x)\oplus \partial_{\bm t}u(x)\bm t}"] & P_{k-1}(e) \arrow[d, "{v(x) \bm t}"]& \\
    \mathbb J^1(x) \arrow[r, "\grad"] & {\mathbb J^0(x) \otimes \mathbb R^2} &
\end{tikzcd}
\end{equation}
Here, we specify the trace structure from faces to vertices. In the first column, the mapping $P_k(f) \to \mathbb J^1(x)$ is defined as $u \mapsto u(x) \oplus \nabla u(x) $; in the second column, the mapping $P_{k-1}(f)\otimes \mathbb R^2 \to \mathbb J^0 (x) \otimes \mathbb R^2$ is defined as $\bm v \mapsto \bm v(x)$. 
\begin{remark}
    Note that here the choice of the trace structure from edges to vertices is not unique. For example, denote by $\bm e_i,1\leq i\leq 2$, the standard basis of $\mathbb R^2$. Fix $i$, then an alternative choice of the trace structure from edges to vertices is that: In the first column, the mapping $P_k(e)\to\mathbb J^1(x)$ is defined as $u\mapsto u(x)\oplus\partial_{\bm t}u(x)\bm e_i$; in the second column, the mapping $P_{k-1}(e)\to\mathbb J^0(x)\otimes\mathbb R^2$ is defined as $v\mapsto v(x)\bm e_i$. The reason we can do this is that we no longer require the composition of traces to be a trace again, and we do not require edge-to-vertex traces to commute with differential operators.
\end{remark}
\begin{proposition}
The FECTS \eqref{eq:fes-HmdR} is compatible with respect to the generalized currents for the de Rham complex \eqref{eq:complex-dR-2}.
\end{proposition}
\begin{proof}
Set $\widetilde{\Upsilon}_{x}:a_0\oplus a_1\mapsto a_0$ for every vertex $x\in\mathcal{V}$, $\widetilde{\Upsilon}_{e}:v\mapsto\int_e v$ for every edge $e\in \mathcal{E}$ and $\widetilde{\Upsilon}_{f}:q\mapsto\int_f q$ for every face $f\in\mathcal{F}$. Then it is easy to see that $\Upsilon_{\bullet}$ can be localized. Note that the face bubble complex is 
\begin{equation*}
    \begin{tikzcd}
        0 \ar[r] &  \mathring P_k^{(1,0)}(f) \ar[r,"\grad"] &\mathring P_{k-1}^{(0,-1)}(\operatorname{rot},f;\mathbb R^2) \ar[r,"\rot"] & P_{k-2}(f) / \mathbb R \ar[r] &0.
    \end{tikzcd}
\end{equation*}
Here the bubble space $\mathring P_{k-1}^{(0,-1)}(\operatorname{rot},f;\mathbb R^2):=\{\bm u\in \mathring P_{k-1}^{(0,-1)}(f)\otimes\mathbb R^2:\bm u\cdot\bm t=0 \text{ on each edge } \ e\subset \partial f\}$. This bubble complex is exact (cf.~\cite{christiansen2018nodal}). The bubble complex on each edge $e\in\mathcal{E}$ is
\begin{equation*}
    \begin{tikzcd}
        0 \ar[r]& \mathring{P}^{(1)}_{k}(e) \ar[r,"\partial_{\bm t}"] &\mathring{P}^{(0)}_{k-1}(e) / \mathbb R \ar[r]& 0,
    \end{tikzcd}
\end{equation*}
and the bubble complex on each vertex $x\in\mathcal{V}$ is isomorphic to
\begin{equation*}
    \begin{tikzcd}
        0 \ar[r]& P_1(\Omega)/\mathbb R\ar[r,"\grad"] &P_0(\Omega) \otimes \mathbb R^2 \ar[r]& 0.
    \end{tikzcd}
\end{equation*}
Both complexes are also exact.
\end{proof}

The corresponding global spaces form the following complex \cite{christiansen2018nodal}:
\begin{equation}
\label{eq:HmdR}
\begin{tikzcd}
	0 \ar[r] & \mathbf{Hm}_k \ar[r,"\grad"] & \mathbf{St}_{k-1} \ar[r,"\rot"]& \mathbf{DG}_{k-2} \ar[r] & 0.
\end{tikzcd}
\end{equation}
Here, $\mathbf{Hm}_k$ is the Hermite element space:
$ \mathbf{Hm}_k = \{ u: u \in P_k(\mathcal{T}), u \in C^1(\mathcal V), u \in C^0(\mathcal E)\}.$
$\mathbf{St}_{k-1}$ is the rotational Stenberg element space:
$\mathbf{St}_{k-1} = \{ \bm u:\bm u \in P_{k-1}(\mathcal{T}) \otimes \mathbb R^2, \bm u \in C^0(\mathcal V), \bm u \cdot \bm t \in C^0(\mathcal E)\}.$
 $\mathbf{DG}_{k-2}$ is the discontinuous piecewise $P_{k-2}$ space.
 As a consequence of Theorem \ref{thm:cohomology}, the cohomology of \eqref{eq:HmdR} is isomorphic to $\mathcal H_{dR}(\Omega)$.
 
\subsection{Falk--Neilan Stokes complex}
Consider the following FECTS with $k\geq 5$, which will lead to a finite element de Rham complex with higher regularity. Such complexes have been used to discretize the Stokes problem, see \cite{falk2013stokes}. 
\begin{equation}
\label{eq:fes-ArdR}
\begin{tikzcd}
    P_k(f) \arrow[r, "\grad"] \arrow[d, "{(u|e,(\partial_{\bm n}u)|_e)}"]  \ar[dd, bend right=60, "", swap] & P_{k-1}(f) \otimes \mathbb R^2 \arrow[r, "\rot"] \arrow[d, "{(\bm v|_e\cdot \bm t,\bm v|_e\cdot \bm n)}"] \ar[dd, bend left=60, "{} "] & P_{k-2}(f)\ar[dd,"w(x)"] \\
    {\begin{pmatrix} P_k(e) \\ P_{k-1}(e)\end{pmatrix}}  \arrow[r, "{\mathbf{d}(\partial_{\bm t}, id)}"] \arrow[d, "{}"] &  {\begin{pmatrix} {P}_{k-1}(e) \\ {P}_{k-1}(e)  \end{pmatrix}}  \arrow[d, ""] &   \\
    \mathbb J^2(x) \arrow[r, "\grad"] &  \mathbb J^1(x) \otimes \mathbb R^2 \ar[r,"\rot"] & \mathbb J^0(x)
\end{tikzcd}
\end{equation}
Here, $\mathbf{d}(\partial_{\bm t}, \operatorname{id})$ denotes $\begin{pmatrix}\partial_{\bm t} & 0 \\ 0& \operatorname{id}  \end{pmatrix}$. Hereafter, we shall adopt this conventional notation to shorten the expression. 
{We first specify the trace structures from faces to vertices: In the first column, the mapping $P_k(f) \to \mathbb J^2(x)$ is defined as $u \mapsto u(x) \oplus \nabla u(x) \oplus \nabla^2u(x)$; in the second column, the mapping $P_{k-1}(f)\otimes\mathbb R^2 \to \mathbb J^1 (x) \otimes \mathbb R^2$ is defined as $\bm v \mapsto \bm v(x)\oplus\nabla \bm v(x)$.

Then we specify the trace structures from edges to vertices: The mapping in the first column $ {\begin{pmatrix} P_k(e) \\ P_{k-1}(e) \end{pmatrix}} \to \mathbb J^2(x) $ is $(a_0, a_1) \mapsto a_0(x)\oplus \big(\partial_{\bm t}a_0(x)\bm t +a_1(x)\bm n\big)\oplus \big(\partial_{\bm t}^2a_0(x)\bm t\bm t^T + \partial_{\bm t}a_1(x)\bm n\bm t^T\big)$; in the second column, the mapping $ {\begin{pmatrix} P_{k-1}(e) \\ P_{k-1}(e) \end{pmatrix}} \to \mathbb J^1(x)\otimes \mathbb R^2 $ is $(a_0, a_1) \mapsto 
\big(a_0(x)\bm t+a_1(x)\bm n\big)\oplus \big(\partial_{\bm t}a_0(x)\bm t\bm t^T + \partial_{\bm t}a_1(x)\bm n\bm t^T\big) $.}

\begin{proposition}
The FECTS \eqref{eq:fes-ArdR} is compatible with respect to the generalized currents for the de Rham complex \eqref{eq:complex-dR-2}.
\end{proposition}
\begin{proof}
Set $\widetilde{\Upsilon}_{x}:a_0\oplus a_1\oplus a_2\mapsto a_0$ for every vertex $x\in\mathcal{V}$, $\widetilde{\Upsilon}_{e}:(v_0,v_1)\mapsto\int_e v_0$ for every edge $e\in \mathcal{E}$ and $\widetilde{\Upsilon}_{f} :q\mapsto\int_fq$ for every face $f\in\mathcal{F}$. Then it is easy to see that $\Upsilon_{\bullet}$ can be localized.  The bubble complex on each face $f\in\mathcal{F}$ is
\begin{equation*}
\begin{tikzcd}
0 \ar[r]&
 \mathring P_{k}^{(2,1)}(f)  
    \arrow[r, "\grad"] 
    & \mathring P_{k-1}^{(1,0)}(f) \otimes \mathbb R^2
    \ar[r,"\rot"] & \mathring P_{k-2}^{(0,-1)}(f) / \mathbb R \ar[r] & 0,
\end{tikzcd}
\end{equation*}
which is exact (cf.~\cite{falk2013stokes}). The bubble complex on each edge $e\in\mathcal{E}$ is 
\begin{equation*}
\begin{tikzcd}[column sep = large]
0 \ar[r]&
 {\begin{pmatrix} \mathring{P}_k^{(2)}(e) \\ \mathring{P}_{k-1}^{(1)}(e) \end{pmatrix}} 
    \arrow[r, "{\mathbf{d}(\partial_{\bm t}, \operatorname{id})}"] 
    & 
 {\begin{pmatrix} \mathring{P}_{k-1}^{(1)}(e) \\ \mathring{P}_{k-1}^{(1)}(e) \end{pmatrix} / \begin{pmatrix} \mathbb R \\ 0 \end{pmatrix}}  
    \ar[r] & 0,
\end{tikzcd}
\end{equation*}
and the bubble complex at each vertex $x\in\mathcal{V}$ is isomorphic to
\begin{equation*}
    \begin{tikzcd}
        0 \ar[r]& P_2(\Omega)/\mathbb R \ar[r,"\grad"]& P_1(\Omega) \otimes \mathbb R^2 \ar[r,"\rot"] &P_0(\Omega)\ar[r] &0,
    \end{tikzcd}
\end{equation*}
which is exact. 
\end{proof}

The corresponding global spaces form the following complex \cite{falk2013stokes}:
\begin{equation}
\label{eq:FNdR}
\begin{tikzcd}
	0 \ar[r] & \mathbf{Ar}_k \ar[r,"\operatorname{grad}"] & \mathbf{Hm}_{k-1} \otimes \mathbb R^2 \ar[r,"\operatorname{rot}"]& \mathbf{DG}_{k-2}^{(0)} \ar[r] & 0.
\end{tikzcd}
\end{equation}
Here, $\mathbf{Ar}_k$ is the Argyris finite element:
$ \mathbf{Ar}_k = \{ u: u \in P_k(\mathcal{T}), u \in C^2(\mathcal V), u \in C^1(\mathcal E)\}.$
The last space $\mathbf{DG}_{k-2}^{(0)}$ is the space of piecewise $P_{k-2}$ functions that are continuous at vertices.   
As a consequence of Theorem \ref{thm:cohomology}, the cohomology of \eqref{eq:FNdR} is isomorphic to $\mathcal H_{dR}(\Omega).$

\subsection{Hu--Zhang Hessian complex}

Now we move our focus beyond the de Rham complex. Consider the following Hessian complex in two dimensions:
\begin{equation}\label{eq:complex-hess-2}\begin{tikzcd} 0 \ar[r,""] &  C^{\infty}(\Omega)  \ar[r,"\hess"] & C^{\infty}(\Omega) \otimes \mathbb S^2 \ar[r,"\rot"] &   C^{\infty}(\Omega) \otimes \mathbb R^2 \ar[r] &0.\end{tikzcd}\end{equation}
The generalized currents are listed in \Cref{lem:currents-2D-hessian}.

Now we consider the following FECTS for the Hessian complex with $k\geq 5$:
\begin{equation}
\label{eq:fes-Hess2D}
\begin{tikzcd}
    P_k(f) \arrow[r, "\hess"] \arrow[d, "{(u|_e,(\partial_{\bm n}u)|_e)}"]  \ar[dd, bend right=60, "", swap] & {P_{k-2}(f)} \otimes \mathbb S^2 \arrow[r, "\rot"] \arrow[d, "{(\bm t^T\bm \sigma|_e\bm t,\bm n^T\bm \sigma|_e\bm t)}"] \ar[dd, bend left=60, "{} "] & P_{k-3}(f) \otimes \mathbb R^2 \\
    {\begin{pmatrix} P_k(e) \\ P_{k-1}(e) \end{pmatrix}} \arrow[r, "{\mathbf{d}(\partial_{\bm t}^2, \partial_{\bm t})}"] \arrow[d, "{}"] & {\begin{pmatrix} P_{k-2}(e) \\ P_{k-2}(e)\end{pmatrix}}  \arrow[d, ""] &   \\
    \mathbb J^2(x) \arrow[r, "\hess"] & \mathbb J^0(x) \otimes \mathbb M^2 & 
\end{tikzcd}
\end{equation}
{We first specify the trace structures from faces to vertices: In the first column, the mapping $P_k(f) \to \mathbb J^2(x)$ is defined as $u \mapsto u(x) \oplus \nabla u(x) \oplus \nabla^2u(x)$; in the second column, the mapping $P_{k-2}(f)\otimes\mathbb S^2 \to \mathbb J^0 (x) \otimes \mathbb M^2$ is defined as $\bm v \mapsto \bm v(x)$. Then we specify the trace structures from edges to vertices: The mapping in the first column $ {\begin{pmatrix} P_k(e) \\ P_{k-1}(e) \end{pmatrix}} \to \mathbb J^2(x) $ is $(a_0, a_1) \mapsto a_0(x) \oplus \big(\partial_{\bm t}a_0(x) \bm t + a_1(x) \bm n\big) \oplus \big(\partial_{\bm t}^2a_0(x) \bm t\bm t^T + \partial_{\bm t}a_1(x)\bm n \bm t^T\big)$; and the mapping in the second column $\begin{pmatrix} P_{k-2}(e) \\ P_{k-2}(e) \end{pmatrix} \to \mathbb J^0(x)\otimes\mathbb M^2 $ is $(a_0, a_1) \mapsto a_0(x) \bm t\bm t^T + a_1(x) \bm n \bm t^T$.

\begin{proposition}
The FECTS \eqref{eq:fes-Hess2D} is compatible with respect to the generalized currents defined in Lemma \ref{lem:currents-2D-hessian} for the Hessian complex \eqref{eq:complex-hess-2}.
\end{proposition}
\begin{proof}
 We first show that $\Upsilon_{\bullet}$ can be localized. Let 
 \begin{equation*}
     \widetilde{\Upsilon}_{f} : P_{k-3}(f) \otimes \mathbb R^2 \to \mathcal{Z}, \quad\bm  v \mapsto \sum_{i=1}^3\big(\int_f\bm v \cdot \psi_i\big)z_i,
 \end{equation*}
 \begin{equation*}
     \widetilde{\Upsilon}_e :{\begin{pmatrix} P_{k-2}(e) \\ P_{k-2}(e)\end{pmatrix}}   \to \mathcal{Z},\quad (a_0,a_1) \mapsto \sum_{i=1}^3\big(\int_e a_0 (\psi_i\cdot \bm t) + a_1(\psi_i \cdot\bm n) \big)z_i,
 \end{equation*}
 and
 \begin{equation*}
     \widetilde{\Upsilon}_{x} : \mathbb J^2(x) \to \mathcal{Z},\quad a_0\oplus a_1\oplus a_2 \mapsto \sum_{i=1}^3\big((\psi_i (x) \cdot a_1) -\tfrac{1}{2} \div \psi_i(x) \cdot a_0\big)z_i,
 \end{equation*}
for each $x\in\mathcal{V},e\in\mathcal{E}$ and $f\in\mathcal{F}$. Therefore, the generalized currents can be localized. The bubble complex on each face $f\in\mathcal{F}$ is
\begin{equation*}
\begin{tikzcd}
0 \ar[r]&
 \mathring P_{k}^{(2,1)}(f)  
    \arrow[r, "\hess"] 
    & \mathring P_{k-2}^{(0,-1)}(\operatorname{rot},f;\mathbb S^2) \ar[r,"\rot"] & P_{k-3}(f)\otimes \mathbb R^2 / \mathcal{RT}^2(f) \ar[r] & 0,
\end{tikzcd}
\end{equation*}
which is exact (cf.~\cite{christiansen2018nodal}). Here the bubble space $\mathring{P}_{k-2}^{(0,-1)}(\operatorname{rot},f;\mathbb S^2):=\{\bm u\in \mathring P_{k-2}^{(0,-1)}(f)\otimes\mathbb S^2:\bm u\bm t=0 \text{ on each edge } \ e\subset \partial f\}$. The bubble complex on edge $e\in\mathcal{E}$ is 
\begin{equation*}
\begin{tikzcd}[column sep = large]
0 \ar[r]&
 {\begin{pmatrix} \mathring{P}_k^{(2)}(e) \\ \mathring{P}_{k-1}^{(1)}(e) \end{pmatrix}} 
    \arrow[r, "{\
    \mathbf{d}(\partial_{\bm t}^2, \partial_{\bm t})}"] 
    & 
 {\begin{pmatrix} \mathring{P}_{k-2}^{(0)}(e) \\ \mathring{P}_{k-2}^{(0)}(e)\end{pmatrix} / \begin{pmatrix} P_1(e) \\ \mathbb R \end{pmatrix}}  
    \ar[r] & 0,
\end{tikzcd}
\end{equation*}
and the bubble complex at each vertex $x\in\mathcal{V}$ is isomorphic to 
\begin{equation*}
\begin{tikzcd}
    0 \ar[r]& P_2(\Omega)/P_1(\Omega) \ar[r,"\hess"]& P_0(\Omega) \otimes \mathbb S^2 \ar[r]& 0.
\end{tikzcd}
\end{equation*}
Both complexes are exact.
\end{proof}
The corresponding global spaces form the following complex:
\begin{equation}
\label{eq:hessiacomplex2d}
\begin{tikzcd}
	0 \ar[r] & \mathbf{Ar}_k \ar[r,"\hess"] & \mathbf{HZ}_{k-2}  \ar[r,"\rot"]& \mathbf{DG}_{k-3} \otimes \mathbb R^2\ar[r] & 0.
\end{tikzcd}
\end{equation}
Here, $\mathbf{HZ}_{k-2}$ is the Hu--Zhang finite element \cite{hu2015}:
$$ \mathbf{HZ}_{k-2} = \{ \bm \sigma: \bm \sigma \in P_{k-2}(\mathcal{T}) \otimes \mathbb S^2, \bm \sigma \in C^0(\mathcal V), \bm \sigma \bm t \in C^0(\mathcal E)\}.$$
As a consequence of Theorem \ref{thm:cohomology}, the cohomology of \eqref{eq:hessiacomplex2d} is isomorphic to $P_1(\Omega) \otimes \mathcal H_{dR} (\Omega).$

\subsection{Hu--Ma--Zhang divdiv+ complex}

Let us finally consider a finite element divdiv complex with additional regularity from \cite{hu2021family}. The example is specially chosen to show that the proposed framework can also handle additional global smoothness, which actually follows the same principle as the Stokes complex above. 

Consider the two-dimensional divdiv complex:
\begin{equation}\label{eq:complex-divdiv-2}\begin{tikzcd} 0 \ar[r,""] &  C^{\infty}(\Omega) \otimes \mathbb R^2 \ar[r,"\sym\curl"] & C^{\infty}(\Omega) \otimes \mathbb S^2 \ar[r,"\div\div"] &   C^{\infty}(\Omega)  \ar[r] &0.
\end{tikzcd}
\end{equation}
The generalized currents for this complex are given in \Cref{lem:currents-2D-divdiv}.

We now introduce the following FECTS with $k\geq 4$:
\begin{equation}
\label{eq:fes-divdiv2D}
\begin{tikzcd}
    P_k(f) \otimes \mathbb R^2 \arrow[r, "\sym\curl"] \arrow[d, "{(\bm u|_e\cdot \bm n, \bm u|_e \cdot \bm t, (\div \bm u)|_e)}", swap]  \ar[dd, bend right=60, "", swap] & {P_{k-1}(f)} \otimes \mathbb S^2 \arrow[r, "\div\div"] \arrow[d, "{(\bm n^T \bm \sigma|_e \bm n, \bm t^T \bm \sigma|_e \bm n,(\div \bm \sigma)|_e \cdot \bm n)}"] \ar[dd, bend left=60, "{} "] & P_{k-3}(f)  \\
    {\begin{pmatrix} P_k(e) \\  P_{k}(e) \\ P_{k-1}(e) \end{pmatrix}} \arrow[r, "
    {\begin{pmatrix} \partial_{\bm t} ~  0 ~ 0 \\ 0 ~  \partial_{\bm t} ~  -\frac{1}{2} \\ 0 ~  0 ~ \frac{1}{2} \partial_{\bm t}\end{pmatrix}  }
    "] \arrow[d, ""] &  {\begin{pmatrix} P_{k-1}(e) \\ P_{k-1}(e) \\ P_{k-2}(e) \end{pmatrix}} \arrow[d, ""] &   \\
    \mathbb J^1(x) \otimes \mathbb R^2 \arrow[r, "\sym \curl"] & \mathbb J^0(x) \otimes \mathbb M^2 & 
\end{tikzcd}
\end{equation}

We first specify the trace structures from faces to vertices: In the first column, the mapping $P_k(f)\otimes\mathbb R^2 \to \mathbb J^1(x)\otimes\mathbb R^2$ is defined as $\bm u \mapsto \bm u(x) \oplus \nabla \bm u(x) $; in the second column, the mapping $P_{k-1}(f)\otimes\mathbb S^2 \to \mathbb J^0 (x) \otimes \mathbb M^2$ is defined as $\bm v \mapsto \bm v(x)$. Then we specify the trace structures from the edges to vertices: The mapping in the first column $ {\begin{pmatrix} P_k(e) \\ P_{k}(e) \\ P_{k-1}(e)\end{pmatrix}} \to \mathbb J^1(x)\otimes\mathbb R^2 $ is $(a_0,a_1,a_2) \mapsto (a_0(x) \bm n + a_1(x) \bm t) \oplus ( \partial_{\bm t}a_0(x)\bm n \bm n^T + \partial_{\bm t}a_1(x) \bm t \bm t^T + a_2(x) \bm n \bm t^T ) $; the mapping in the second column $ {\begin{pmatrix} P_{k-1}(e) \\ P_{k-1}(e) \\ P_{k-2}(e)\end{pmatrix}} \to \mathbb J^0(x)\otimes\mathbb M^2 $ is $(a_0,a_1,a_2) \mapsto a_0(x)\bm n\bm n^T + a_1(x)\bm n\bm t^T$.

The operators on the edge are defined according to the following lemma, cf.~\cite{hu2021family}. 
\begin{lemma}
If $\bm \sigma = \sym \curl \bm u$, then 
\begin{equation*}
    \bm n^T\bm \sigma \bm n = \partial_{\bm t}(\bm u \cdot \bm n),
\end{equation*}
\begin{equation*}
    \bm t^T \bm \sigma \bm n = \partial_{\bm t}(\bm u \cdot \bm t) - \frac{1}{2} \div \bm u,
\end{equation*}
\begin{equation*}
    \div \bm \sigma \cdot \bm n = \frac{1}{2} \partial_{\bm t}(\div \bm u).
\end{equation*}
	
\end{lemma}
\begin{proposition}
The FECTS \eqref{eq:fes-divdiv2D} is compatible with respect to the generalized currents defined in Lemma \ref{lem:currents-2D-divdiv} for the divdiv complex \eqref{eq:complex-divdiv-2}.
\end{proposition}
\begin{proof}
We first show that $\Upsilon_{\bullet}$ can be localized. Let
\begin{equation*}
    \widetilde{\Upsilon}_{x} :\mathbb J^1(x)\otimes\mathbb R^2\to \mathcal{Z},\quad (a_0, a_1) \mapsto \sum_{i=1}^3\big( \frac{1}{2} (\operatorname{Tr}( a_1) \psi_i)(x)-(a_0 \cdot \nabla \psi_i)(x)\big)z_i,
\end{equation*}
\begin{equation*}
    \widetilde{\Upsilon}_e :{\begin{pmatrix} {P}_{k-1}(e) \\ {P}_{k-1}(e)\\ P_{k-2}(e) \end{pmatrix}}   \to \mathcal Z, \quad (a_0,a_1,a_2) \mapsto\sum_{i=1}^3\big( \int_e  a_2 \psi_i -a_0 (\nabla \psi_i\cdot \bm n) - a_1(\nabla \psi_i \cdot\bm t)\big)z_i,
\end{equation*}
\begin{equation*}
    \widetilde{\Upsilon}_{f} : P_{k-3}(f) \to \mathcal Z,\quad \widetilde{\Upsilon}_f(q) = \sum_{i=1}^3\big(\int_f  q \psi_i\big)z_i,
\end{equation*}
for each $x\in\mathcal{V},e\in\mathcal{E}$ and $f\in\mathcal{F}$. Thus, the generalized currents can be localized.

The bubble complex on face $f$ is
\begin{equation*}
\begin{tikzcd}
0 \ar[r]&
 \mathring P_{k}^{(1,0)\cap\operatorname{div}}(f;\mathbb R^2)  
    \arrow[r, "\sym\curl"] 
    &\mathring P_{k-1}^{(0,-1)\cap\operatorname{div}}(\operatorname{divdiv},f;\mathbb S^2) \ar[r,"\div\div"] & P_{k-3}(f) / P_1(f) \ar[r] & 0,
\end{tikzcd}
\end{equation*}
which is exact (cf.~\cite{hu2021family}). Here the bubble space $\mathring P_{k}^{(1,0)\cap\operatorname{div}}(f;\mathbb R^2)  := \{\bm u\in \mathring P_{k}^{(1,0)}(f)\otimes\mathbb R^2:\operatorname{div}\bm u\in \mathring{P}_{k-1}^{(0,0)}(f)\}$ and the space $\mathring P_{k-1}^{(0,-1)\cap\operatorname{div}}(\operatorname{divdiv},f;\mathbb S^2):=\{\bm \sigma\in \mathring P_{k-1}^{(0,-1)}(f)\otimes \mathbb S^2:\bm \sigma\bm n =0,\operatorname{div}\bm\sigma\cdot\bm n=0\text{ on }\partial f\}$. 
The bubble complex on edge $e\in \mathcal{E}$ is 
\begin{equation*}
\begin{tikzcd}[column sep = 2.3em]
0\ar[r]&
 {\begin{pmatrix} \mathring{P}_k^{(1)}(e) \\  \mathring{P}_{k}^{(1)}(e) \\ \mathring{P}_{k-1}^{(0)}(e) \end{pmatrix}
 }
  \ar[r,"
    {\begin{pmatrix} \partial_{t} ~  0 ~ 0 \\ 0 ~  \partial_{\bm t} ~  -\tfrac{1}{2} \\ 0 ~  0 ~ \frac{1}{2} \partial_{\bm t}\end{pmatrix}  }
    "]
   &
 {\begin{pmatrix} \mathring{P}_{k-1}^{(0)}(e)\\ \mathring{P}_{k-1}^{(0)}(e) \\ P_{k-2}(e)\end{pmatrix} / \left\{ \begin{pmatrix}
 a \\ \partial_{\bm t}b \\ -b
\end{pmatrix}:a\in\mathbb R,b \in P_1(e) \right\}
}
 \ar[r]&0.
\end{tikzcd}
\end{equation*}
The quotient space here is defined as
\begin{align*}
    {\begin{pmatrix} \mathring{P}_{k-1}^{(0)}(e)\\ \mathring{P}_{k-1}^{(0)}(e) \\ P_{k-2}(e)\end{pmatrix} / \{ \begin{pmatrix}
 a\\ \partial_{\bm t}b \\ -b
\end{pmatrix}}:a\in &\mathbb R, b \in P_1(e) \} = \{(u, v,w)^T\in (\mathring{P}_{k-1}^{(0)}(e), \mathring{P}_{k-1}^{(0)}(e), P_{k-2}(e))^T\\ &:\int_e (ua + v\partial_{\bm t}b-wb)\mathrm{d}x=0,\text{ for all }b\in P_1(e),a\in P_0(e)\}.
 \end{align*}
To verify that the sequence forms a complex, it is sufficient to perform integration by parts. Together with dimension counting, this shows that the edge bubble complex is exact. Finally, the bubble complex at vertex $x\in\mathcal{V}$ is isomorphic to 
\begin{equation*}
    \begin{tikzcd}
        0 \ar[r]& P_1(\Omega) \otimes\mathbb R^2/ \mathcal{RT}^2\ar[r,"\sym\curl"] & P_0(\Omega) \otimes \mathbb S^2 \ar[r] &0,
    \end{tikzcd}
\end{equation*}
which is also exact.
\end{proof}
The corresponding global spaces form the following complex \cite{hu2021family}:
\begin{equation}
\label{eq:divdivcomplex2d}
\begin{tikzcd}
	0 \ar[r] & \mathbf{Hm}^{\div}_k \ar[r,"\sym\curl"] & \mathbf{HMZ}_{k-1}  \ar[r,"\div\div"]& \mathbf{DG}_{k-3}\ar[r] & 0.
\end{tikzcd}
\end{equation}
Here,
$ \mathbf{Hm}^{\div}_k = \{ \bm u: \bm u \in P_k(\mathcal{T}) \otimes \mathbb R^2, \bm u \in C^1(\mathcal V), \bm u,\div \bm u \in C^0(\mathcal E)\}$ has $H^1(\div)$ global continuity, 
and $\mathbf{HMZ}_{k-1}$ is the Hu--Ma--Zhang $H(\div) \cap H(\div\div)$ finite element:
\begin{equation*}
\mathbf{HMZ}_{k-1}: = \{ \bm \sigma : \bm \sigma \in P_{k-1}(\mathcal{T}) \otimes \mathbb S^2, \bm \sigma \in C^0(\mathcal V), \bm \sigma \bm n , \div \bm \sigma \cdot \bm n\in C^0(\mathcal E)\}.
\end{equation*}
As a consequence of Theorem \ref{thm:cohomology}, the cohomology of \eqref{eq:divdivcomplex2d} is isomorphic to $\mathcal{RT}^2 \otimes \mathcal H_{dR}(\Omega)$.  
\section{Finite element complexes with trace structures: 3D case }
\label{sec:3d}
This section deals with complexes in three-dimensional spaces. Note that for the examples presented herein, the exactness of all cell bubble complexes was proved when the finite element complexes were introduced. Consequently, we omit the cell bubble complexes. 
Moreover, some face bubbles were introduced as an intermediate step of the construction. However, some of them are not suitable for deriving compatibility conditions. Therefore, we present the bubble complexes that can be produced by the compatibility conditions, without providing detailed proofs. 
Edge bubble complexes appear to be less widely discussed in the literature, and their exactness can be easily verified through dimension counting.

We first introduce some standard differential operators: Given a scalar function $u$ and vector function $\bm q = (q_1,q_2,q_3)^T$, the differential operators are defined by
\begin{equation*}
    \operatorname{grad}u =\nabla u = (\partial_x u,\partial_y u,\partial_z u)^T,
\end{equation*}
\begin{equation*}
    \operatorname{curl}\bm q = \nabla\times\bm q = (\partial_yq_3-\partial_zq_2,\partial_zq_1-\partial_xq_3,\partial_xq_2-\partial_yq_1)^T,
\end{equation*}
\begin{equation*}
    \operatorname{div}\bm q =\nabla\cdot\bm q = \partial_xq_1+\partial_yq_2+\partial_zq_3.
\end{equation*}
For vector and tensor functions, these differential operators act row-wise. For a symmetric matrix field $\bm \sigma$, define
$
\operatorname{inc}\bm \sigma
=
\operatorname{curl}\bigl((\operatorname{curl}\bm \sigma)^T\bigr)
$. We can also define the Hessian operator
\begin{equation*}
    \hess u =\nabla^2u =\begin{pmatrix}
        \partial_{x}^2u & \partial_{x}\partial_yu & \partial_{x}\partial_zu\\
        \partial_{y}\partial_xu & \partial_y^2u & \partial_{y}\partial_zu\\
\partial_{z}\partial_xu & \partial_{z}\partial_yu & \partial_z^2u
    \end{pmatrix}   .
\end{equation*}

In this section, we need to discuss differential operators on surfaces: For a given face $f$, choose orthonormal tangent vectors $\bm t_1, \bm t_2$ and the unit normal vector $\bm n$ such that $\bm t_1 \times \bm t_2 = \bm n$. Let $\Pi_f(\bm u) = \begin{pmatrix} \bm u\cdot \bm t_1 \\ \bm u \cdot \bm t_2\end{pmatrix}$ be the projection to the face tangential space. Given a scalar function $u$ and vector function $\bm q =(q_1,q_2)^T$, we can then define the surface differential operators on face $f$ by
\begin{align*}
    \operatorname{grad}_f u =(\partial_{\bm t_1}u,\partial_{\bm t_2}u)^T,&\quad \operatorname{curl}_f u = (-\partial_{\bm t_2}u,\partial_{\bm t_1}u)^T,\\
    \operatorname{div}_f \bm q = \partial_{\bm t_1}q_1 + \partial_{\bm t_2}q_2,&\quad \operatorname{rot}_f \bm q = \partial_{\bm t_1}q_2 - \partial_{\bm t_2}q_1,
\end{align*}
and 
\begin{align*}
        \hess_f u =\nabla_f^2u =\begin{pmatrix}
        \partial_{\bm t_1}^2u & \partial_{\bm t_1}\partial_{\bm t_2}u\\
        \partial_{\bm t_1}\partial_{\bm t_2}u &\partial_{\bm t_2}^2u
    \end{pmatrix}.
\end{align*}
For vector and tensor functions, the surface differential operators act row-wise.

For a given edge, we use $\bm t$ to denote its unit tangential vector, and $\bm n_1, \bm n_2$ to denote its two orthogonal unit normal vectors, such that $\bm n_1 \times \bm n_2 = \bm t$. Let $N_e(\bm u) = \begin{pmatrix} \bm u\cdot \bm n_1 \\ \bm u \cdot \bm n_2\end{pmatrix}$ be the projection to the edge normal space. For the matrix case, $\Pi_f, N_e$ on the left (right, resp.) denote the projection row-wise (column-wise, resp.). 
When considering a face $f$ and one of its edges $e$, we still use $\bm t$ to denote the tangential vector of $e$ and $\bm n_{f,e} = \bm t \times \bm n$ to denote the edge-face normal vector. 

Finally, denote by $\{\bm e_1,\bm e_2,\bm e_3\}$ the standard basis of $\mathbb R^3$. Then given a two-dimensional vector $\bm v = (v_1,v_2)^T$, we can define the extension operator $E_3(\bm v)=v_1\bm e_1 +v_2\bm e_2$. The definition of $E_3$ can be extended naturally to high order tensors on a two-dimensional face. Namely, $E_3$ induces a mapping from $(\mathbb R^2)^{\otimes n}$ to $(\mathbb R^3)^{\otimes n}$.
\subsection{Hu--Liang Hessian complex}
This subsection focuses on the Hessian complex in three dimensions:
	\begin{equation}\label{eq:complex:hess-3}
    \begin{tikzcd} 0 \ar[r,""] &  C^{\infty}(\Omega)\ar[r,"\hess"] & C^{\infty}(\Omega) \otimes \mathbb S^3  \ar[r,"\curl"] & C^{\infty}(\Omega) \otimes \mathbb T^3 \ar[r,"\div"] &   C^{\infty}(\Omega) \otimes \mathbb R^3 \ar[r] &0.\end{tikzcd}\end{equation}
The generalized currents of this complex are given in \Cref{lem:currents-3d-hessian}. 

We now introduce the following FECTS with $k\geq 9$, which is modified from \cite{hu2021conforming}. 

\par

\begin{equation}
\label{eq:fes-Hess3D}
\begin{tikzcd}
    P_k(T) \arrow[r, "\hess"] \arrow[d, "{(u|_f,(\partial_{\bm n}u)|_f)}"]  & {P_{k-2}(T)} \otimes \mathbb S^3 \arrow[r, "\curl"] \arrow[d, "{ (\Pi_f \bm \sigma|_f  \Pi_f, \bm n^T(\Pi_f \bm \sigma|_f) )}"]  & P_{k-3}(T) \otimes \mathbb T^3 \ar[r,"\div"] \arrow[d, "{(\Pi_f (\bm \tau|_f \bm n), \bm n^T \bm \tau|_f \bm n)}"] & P_{k-4}(T) \otimes \mathbb R^3 \arrow[ddd, "{\bm v(x)}"]
    \\
    {\begin{pmatrix} P_k(f) \\ P_{k-1}(f) \end{pmatrix}} 
    \arrow[r, "{\mathbf{d}(\hess_f, \grad_f)}"] 
    \arrow[d, "{}"] 
    &  
    {\begin{pmatrix} P_{k-2}(f) \otimes \mathbb S^2\\ P_{k-2}(f) \otimes \mathbb R^2 \end{pmatrix}}  
    \arrow[d, ""] 
    \ar[r,"{{\mathbf{d}(\rot_f, \rot_f)}}"] &  
     {\begin{pmatrix} P_{k-3}(f) \otimes \mathbb R^2\\ P_{k-3}(f) \end{pmatrix}} \ar[dd,""]  \\
    {\begin{pmatrix} P_k(e) \\ P_{k-1}(e) \otimes \mathbb R^2 \\ P_{k-2}(e) \otimes \mathbb M^2\end{pmatrix}} \ar[r,"{\mathbf{d}(\partial_{\bm t}^2,\partial_{\bm t},\operatorname{id})}"] \ar[d] &  {\begin{pmatrix} P_{k-2}(e) \\ P_{k-2}(e) \otimes \mathbb R^2 \\ P_{k-2}(e) \otimes \mathbb M^2\end{pmatrix}}\ar[d] & &\\
    \mathbb J^4(x) \arrow[r, "\hess"] & \mathbb J^2(x) \otimes \mathbb M^3 \arrow[r, "\curl"] & \mathbb J^1(x) \otimes \mathbb M^3 \arrow[r, "\div"]& \mathbb J^0(x) \otimes \mathbb R^3
\end{tikzcd}
\end{equation}
For brevity, the remaining maps are listed in Appendix \ref{sec:mapping}.

\begin{proposition}
The FECTS \eqref{eq:fes-Hess3D} is compatible with respect to the generalized currents defined in Lemma \ref{lem:currents-3d-hessian} for the Hessian complex \eqref{eq:complex:hess-3}.
\end{proposition}
\begin{proof}
    
We first verify the localization property. Let
\begin{equation*}
    \widetilde{\Upsilon}_T: P_{k-4}(T)\otimes \mathbb R^3\rightarrow \mathcal{Z},\quad \bm q\mapsto \sum_{i=1}^4\big(\int_{T} \bm q \cdot \psi_i\big)z_i
\end{equation*}
and
\begin{equation*}
    \widetilde{\Upsilon}_f:   {\begin{pmatrix} P_{k-3}(f) \otimes \mathbb R^2\\ P_{k-3}(f)  \end{pmatrix}} \to \mathcal{Z} ,\quad  (a_0,a_1) \mapsto \sum_{i=1}^4\big(\int_{f} a_0 \cdot (\Pi_f\psi_i) + a_1 (\psi_i\cdot \bm n_f)\big)z_i,
\end{equation*}
and
\begin{equation*}
     \widetilde{\Upsilon}_e : {\begin{pmatrix} P_{k-2}(e) \\ P_{k-2}(e) \otimes \mathbb R^2 \\ P_{k-2}(e) \otimes \mathbb M^2\end{pmatrix}} \to \mathcal{Z} ,\quad (a_0,a_1,a_2) \mapsto \sum_{i=1}^4\big(\int_e a_0 (\psi_i \cdot \bm t) + a_1 \cdot (N_e\psi_i)\big)z_i ,
\end{equation*}
and
\begin{equation*}
    \widetilde{\Upsilon}_x: \mathbb J^4 (x)\to \mathcal{Z},\quad a_0\oplus a_1\oplus a_2\oplus a_3\oplus a_4\mapsto \sum_{i=1}^4\big( (\psi_i \cdot a_1)(x) - \tfrac{1}{3}(\div \psi_i \cdot a_0)(x)\big)z_i,
\end{equation*}
for each $T\in\mathcal{K},f\in\mathcal{F},e\in\mathcal{E}$ and $x\in\mathcal{V}$. Then the generalized currents can be localized.

The cell bubble complex is given in \cite{hu2021conforming}. 
The bubble complex on each face is:
\begin{equation*}
\begin{tikzcd}
0 \arrow[r]
&
{\begin{pmatrix}
    \mathring{P}_k^{(4,2)}(f) \\
    \mathring{P}_{k-1}^{(3,1)}(f)
\end{pmatrix}}
\arrow[r, "{\mathbf{d}(\hess_f,\grad_f)}"]
&
{\begin{pmatrix}
    \mathring{P}_{k-2}^{(2,0)}(f) \otimes \mathbb S^2 \\
    \mathring{P}_{k-2}^{(2,0)}(f) \otimes \mathbb R^2
\end{pmatrix}}
\arrow[d, "{\mathbf{d}(\rot_f,\rot_f)}"']
\\
&&
{\begin{pmatrix}
    \mathring{P}_{k-3}^{(1,-1)}(f) \otimes \mathbb R^2 \\
    \mathring{P}_{k-3}^{(1,-1)}(f) 
\end{pmatrix}
\big/
\begin{pmatrix}
    RT(f) \\
    \mathbb R
\end{pmatrix}}
\arrow[r]
&
0,
\end{tikzcd}
\end{equation*}
where $RT(f):=\{\bm a+b\Pi_f\bm x,a\in\mathbb R^2,b\in\mathbb R\}$. The bubble complex on each edge is:
\begin{equation*}
\begin{tikzcd}[column sep = large]
   0\ar[r] &{\begin{pmatrix} \mathring{P}_k^{(4)}(e) \\ \mathring{P}_{k-1}^{(3)}(e) \otimes \mathbb R^2 \\ \mathring{P}_{k-2}^{(2)}(e) \otimes \mathbb S^2\end{pmatrix}} \ar[r,"{\mathbf{d}(\partial_{\bm t}^2,\partial_{\bm t},\operatorname{id})}"] &  {\begin{pmatrix} \mathring{P}_{k-2}^{(2)}(e)\\ \mathring{P}_{k-2}^{(2)}(e) \otimes \mathbb R^2 \\ \mathring{P}_{k-2}^{(2)}(e) \otimes \mathbb S^2\end{pmatrix} / \begin{pmatrix} P_1(e) \\ \mathbb R^2 \\ 0\end{pmatrix}}\ar[r] &0.
   \end{tikzcd}
\end{equation*}
The bubble complex at each vertex is isomorphic to
\begin{equation*}
\begin{tikzcd}
	0 \ar[r] & P_4(\Omega)/P_1(\Omega) \ar[r,"\hess"] & P_2(\Omega)\otimes\mathbb S^3 \ar[r,"\curl"]&  P_1(\Omega)\otimes\mathbb T^3\ar[r,"\div"] & P_0(\Omega) \otimes \mathbb R^3 \ar[r] & 0.
\end{tikzcd}
\end{equation*}
Their exactness follows from the injectivity of the first differential operator and dimension counting.
\end{proof}
The corresponding global spaces form the following complex \cite{hu2021conforming}:
\begin{equation}
\begin{tikzcd}
\label{eq:hessain3d}
	0 \ar[r] & \mathbf{Z}_k \ar[r,"\hess"] & \mathbf{HL}^{\curl}_{k-2}  \ar[r,"\curl"]&  \mathbf{HL}^{\div}_{k-3}\ar[r,"\div"] & \mathbf{DG}_{k-4}^{(0)} \otimes \mathbb R^3 \ar[r] & 0.
\end{tikzcd}
\end{equation}

Here, each space in this complex is specified as follows:
$\mathbf{Z}_k$ is the \v{Z}en\'{i}\v{s}ek--Zhang $C^1$ finite element in three dimensions:
\begin{equation*}
	\mathbf{Z}_k := \{ u : u \in P_k(\mathcal{T}), u\in C^4(\mathcal V) \cap C^2(\mathcal E) \cap C^1(\mathcal F)\}.
\end{equation*}
$\mathbf{HL}_{k-2}^{\curl}$ is the Hu--Liang $H(\curl;\mathbb S^3)$ finite element:
\begin{equation*}
\mathbf{HL}_{k-2}^{\curl}:= \{ \bm \sigma : \bm \sigma  \in P_{k-2}(\mathcal{T}) \otimes \mathbb S^3 , \bm \sigma \in C^2(\mathcal V) \cap C^0(\mathcal E), \Pi_f\bm\sigma \in C^0(\mathcal F)\}.
\end{equation*}
$\mathbf{HL}_{k-3}^{\div}$ is the Hu--Liang $H(\div;\mathbb T^3)$ finite element:
\begin{equation*}
\mathbf{HL}_{k-3}^{\div} := \{ \bm \tau : \bm \tau \in P_{k-3}(\mathcal{T}) \otimes \mathbb T^3 , \bm \tau \in C^1(\mathcal V), \bm \tau \bm n \in C^0(\mathcal F)\}.
\end{equation*}
$\mathbf{DG}_{k-4}^{(0)}\otimes\mathbb R^3$ is defined by
\begin{equation*}
    \mathbf{DG}_{k-4}^{(0)}\otimes\mathbb R^3 :=\{\bm q:\bm q\in P_{k-4}(\mathcal{T})\otimes\mathbb R^3,\bm q\in C^0(\mathcal{V})\}.
\end{equation*}
As a consequence, the cohomology of \eqref{eq:hessain3d} is isomorphic to $P_1(\Omega) \otimes \mathcal H_{dR}(\Omega).$

\section{Conclusions}\label{sec:conclu}
In this paper, we have introduced a unified framework of finite element complexes with trace structures (FECTS). The main idea is to construct a direct sum decomposition of finite element complexes into exact local bubble subcomplexes and a skeletal subcomplex carrying the cohomology. This decomposition provides a common approach to characterizing the cohomology of finite element complexes and constructing local, uniformly $L^2$-bounded commuting projections. We have illustrated the framework through several two- and three-dimensional finite element complexes. A comprehensive treatment of all relevant complexes is beyond the scope of this paper. In particular, applying the framework to other complexes, such as three-dimensional divdiv and elasticity complexes, requires more delicate technical arguments.

There are several possible extensions of the present framework. First, we have considered only the case without imposed boundary conditions. The framework can also be extended to complexes with boundary conditions imposed on the whole boundary or on part of it, as suggested by recent work on finite element Stokes complexes with partial boundary conditions \cite{brubeck2026cohomology}. Second, FECTS does not require the same finite element spaces on every cell. The framework allows cell-dependent shape spaces, provided that their trace images agree on shared subsimplices and the stated compatibility assumptions hold. Related results for two-dimensional stress complexes can be found in \cite{aznaran2024uniformly}. A systematic theory of interpolation and approximation within this general framework is another direction for future research.

\section*{Declaration of AI use}
The main content of this paper, including all mathematical results and proofs, was developed independently by the authors. After completing the initial draft, the authors used ChatGPT (GPT-6) to check for typographical, spelling, and grammatical errors and to identify additional relevant references. All suggested edits and references were reviewed and verified by the authors, who take full responsibility for the final manuscript.
\appendix

\section{$L^2$ boundedness of the interpolation operator}\label{sec:L2bound}
\begin{lemma}\label{lem:L2estimate:basis}
    For any $\phi_{\tau,k}^i$ defined in Lemma \ref{lem:ortho_basis}, if the differential operator $d^k$ is of order $l$, then
    \begin{equation*}
        \tfrac{1}{\hat h_{\tau}^l}\|\mathsf E_{\tau}^k(\phi_{\tau,k}^i)\|_{L^2(\st^1(\tau))} +  \|\mathsf E_{\tau}^{k+1}(d\phi_{\tau,k}^i)\|_{L^2(\st^1(\tau))}\leq C \hat h_{\tau}^{(n-\operatorname{dim}\tau )/2},
    \end{equation*}
    with $\hat{h}_{\tau} = \operatorname{diam}(\operatorname{star}(\tau))$.
\end{lemma}
\begin{proof}
    Following the scaling argument in \cite{arnold2006finite}, let $x_0$ denote the first vertex of $\tau$, and define $\varphi(x) = (x-x_0)/\hat h_{\tau}$ which maps $\tau$ to $\hat{\tau} = \varphi(\tau)$ and $\st^1(\tau)$ to $\st^1(\hat{\tau}) = \varphi(\st^1(\tau))$. Define $\hat{\phi}_{\tau,k}^i = \varphi^{-1*}\phi_{\tau,k}^i$ and it is easy to check that $\mathsf E_{\hat{\tau}}^k(\hat{\phi}_{\tau,k}^i) = \varphi^{-1*}\mathsf E_{\tau}^k(\phi_{\tau,k}^i)$. Then
    \begin{equation*}
        \|\mathsf E_{\tau}^k(\phi_{\tau,k}^i)\|_{L^2(\st^1(\tau))} = \hat h^{n/2}_{\tau}\|\mathsf E_{\hat{\tau}}^k(\hat{\phi}_{\tau,k}^i)\|_{L^2(\st^1(\hat{\tau}))}.
    \end{equation*}
    Note that 
    \begin{align*}
        \langle \hat{\phi}_{\tau,k}^i,   \hat{\phi}_{\tau,k}^i \rangle_{\tilde{B}^k(\hat{\tau})} &= (d\hat{\phi}_{\tau,k}^i, d\hat{\phi}_{\tau,k}^i)_{L^2(\hat{\tau})} \\
        & = \hat h_{\tau}^{2l-\operatorname{dim}\tau} (d{\phi_{\tau,k}^i}, d{\phi_{\tau,k}^i})_{L^2({\tau})} = \hat h_{\tau}^{2l-\operatorname{dim}\tau}.
    \end{align*} 
    Hence
    \begin{equation*}
        \|\mathsf E_{\hat{\tau}}^k(\hat{\phi}_{\tau,k}^i)\|_{L^2(\st^1(\hat{\tau}))} \leq C(\st^1(\hat{\tau})) \hat h_{\tau}^{l-\operatorname{dim}\tau/2} ,
    \end{equation*}
    where $C(\st^1(\hat{\tau}))$ is a continuous function of $\st^1(\hat{\tau})$. The mesh regularity assumption implies that there exists a constant $C$ independent of $\tau$ and $\hat h_{\tau}$ such that $C(\st^1(\hat{\tau}))\leq C$. Then we get that
    \begin{equation*}
        \|\mathsf E_{\tau}^k(\phi_{\tau,k}^i)\|_{L^2(\st^1(\tau))}\leq C\hat h_{\tau}^{l+(n-\operatorname{dim}\tau)/2}.
    \end{equation*}
    Similarly
    \begin{equation*}
        \|\mathsf E_{\tau}^{k+1}(d\phi_{\tau,k}^i)\|_{L^2(\st^1(\tau))}\leq C \hat h_{\tau}^{(n-\operatorname{dim}\tau )/2}.
    \end{equation*}
\end{proof}
\section{Remaining mappings in \eqref{eq:fes-Hess3D}}\label{sec:mapping}
We specify the remaining mappings in \eqref{eq:fes-Hess3D} as follows. For convenience, we use $(I,J)$ to denote the entry in the $I$-th row and $J$-th column. Here the other mappings in the first column are:
\begin{itemize}
    \item[-] (1,1) to (3,1): from $P_k(T)$ to ${\begin{pmatrix} P_k(e) \\ P_{k-1}(e) \otimes \mathbb R^2 \\ P_{k-2}(e) \otimes \mathbb M^2\end{pmatrix}}$ is $ u\mapsto
    \begin{pmatrix} u|_e \\ [(\partial_{\bm n_i} u)|_e]_i \\  [(\partial_{\bm n_{i}}\partial_{\bm n_{j}}u)|_e]_{ij} \end{pmatrix}$ with $i,j=1,2$;
    \item[-] (1,1) to (4,1): from $P_k(T)$ to $\mathbb{J}^4(x)$ is $u\mapsto u(x)\oplus\nabla u(x) \oplus\nabla^2 u(x) \oplus\nabla^3u(x)\oplus \nabla^4u(x)$;
        \item[-] (2,1) to (3,1): from $\begin{pmatrix} P_k(f) \\ P_{k-1}(f) \end{pmatrix}$ to ${\begin{pmatrix} P_k(e) \\ P_{k-1}(e) \otimes \mathbb R^2 \\ P_{k-2}(e) \otimes \mathbb M^2\end{pmatrix}}$ is 
        \begin{equation*}
            (a_0,a_1) \mapsto  \begin{pmatrix} a_0|_e \\ \partial_{\bm n_{f,e}} a_0|_e\Pi_f(\bm t)+ a_1|_e  \Pi_f(\bm n_{f,e}) \\ \partial_{\bm n_{f,e}}^2 a_0|_e\Pi_f(\bm t)\Pi_f(\bm  n_{f,e})^T+ \partial_{\bm n_{f,e}} a_1|_e  \Pi_f(\bm n_{f,e})\Pi_f(\bm n_{f,e})^T\end{pmatrix} ;
        \end{equation*}
        \item[-] (2,1) to (4,1): from $\begin{pmatrix} P_k(f) \\ P_{k-1}(f) \end{pmatrix}$ to $\mathbb J^4(x)$ is $(a_0,a_1)\mapsto a_0(x)\oplus (E_3\nabla_fa_0(x) + a_1(x)\bm e_3)\oplus (E_3\nabla_f^2a_0(x)+E_3\nabla_fa_1(x)\otimes\bm e_3)\oplus (E_3\nabla_f^3a_0(x)+E_3\nabla_f^2a_1(x)\otimes\bm e_3)\oplus (E_3\nabla_f^4a_0(x)+E_3\nabla_f^3a_1(x)\otimes\bm e_3)$;
        \item[-] (3,1) to (4,1): from ${\begin{pmatrix} P_k(e) \\ P_{k-1}(e) \otimes \mathbb R^2 \\ P_{k-2}(e) \otimes \mathbb M^2\end{pmatrix}}$ to $\mathbb J^4(x)$ is $(a_0,a_1,a_2)\mapsto a_0(x) \oplus (\partial_{\bm t}a_0(x)\bm e_3+ E_3a_1(x))\oplus (\partial_{\bm t}^2a_0(x) \bm e_3^{\otimes 2} + E_3\partial_{\bm t}a_1(x)\otimes \bm e_3 + E_3a_2(x) ) \oplus (\partial_{\bm t}^3a_0(x) \bm e_3^{\otimes 3} + E_3\partial_{\bm t}^2a_1(x)\otimes \bm e_3^{\otimes 2} + E_3\partial_{\bm t}a_2(x)\otimes \bm e_3 )  \oplus (\partial_{\bm t}^4a_0(x) \bm e_3^{\otimes 4} + E_3\partial_{\bm t}^3a_1(x)\otimes \bm e_3^{\otimes 3} + E_3\partial_{\bm t}^2a_2(x) \otimes \bm e_3^{\otimes 2})$, where $\bm e_3^{\otimes k}$ represents the $k$-th tensor power of the vector $\bm e_3$.
\end{itemize}
The other mappings in the second column are:
\begin{itemize}
    \item[-] (1,2) to (3,2): from $P_{k-2}(T)\otimes\mathbb S^3$ to $\begin{pmatrix} P_{k-2}(e) \\ P_{k-2}(e) \otimes \mathbb R^2 \\ P_{k-2}(e) \otimes \mathbb M^2\end{pmatrix}$ is $\bm \sigma\mapsto \begin{pmatrix} \bm t^T \bm \sigma|
    _e\bm t \\  [\bm n_i^T \bm \sigma|_e \bm t]_i \\  [\bm n_i^T \bm \sigma|_e\bm n_j]_{ij}\end{pmatrix}$ with $i,j=1,2$;
    \item [-] (1,2) to (4,2): from $P_{k-2}(T)\otimes\mathbb S^3$ to $\mathbb J^2(x)\otimes\mathbb M^3$ is $\bm\sigma\mapsto \bm\sigma(x)\oplus\nabla\bm\sigma(x)\oplus\nabla^2\bm\sigma(x)$;
    \item[-] (2,2) to (3,2): from $\begin{pmatrix} P_{k-2}(f) \otimes \mathbb S^2\\ P_{k-2}(f) \otimes \mathbb R^2 \end{pmatrix}$ to $\begin{pmatrix} P_{k-2}(e) \\ P_{k-2}(e) \otimes \mathbb R^2 \\ P_{k-2}(e) \otimes \mathbb M^2\end{pmatrix}$ is 
    \begin{equation*}
        (a_0,a_1) \mapsto  \begin{pmatrix}  \Pi_f(\bm t)^T a_0|_e \Pi_f(\bm t) \\ (\Pi_f(\bm t)^T a_0|_e \Pi_f(\bm n_{f,e}))\Pi_f(\bm t) +  (\Pi_f(\bm t)^Ta_1|_e)\Pi_f(\bm n_{f,e}) \\ (\Pi_f(\bm n_{f,e})^T a_0|_e \Pi_f(\bm n_{f,e}))\Pi_f(\bm t)\Pi_f(\bm n_{f,e})^T +  (\Pi_f(\bm n_{f,e})^Ta_1|_e)\Pi_f(\bm n_{f,e}) \Pi_f(\bm n_{f,e})^T \end{pmatrix};
    \end{equation*}

    \item [-] (2,2) to (4,2): from $\begin{pmatrix} P_{k-2}(f) \otimes \mathbb S^2\\ P_{k-2}(f) \otimes \mathbb R^2 \end{pmatrix}$  to $\mathbb{J}^2(x)\otimes\mathbb M^3$ is  $(a_0,a_1)\mapsto (E_3a_0(x)+E_3a_1(x)\otimes\bm e_3)\oplus (E_3\nabla_fa_0(x)+E_3\nabla_fa_1(x)\otimes\bm e_3)\oplus (E_3\nabla_f^2a_0(x)+E_3\nabla_f^2a_1(x)\otimes\bm e_3)$;
    \item [-] (3,2) to (4,2): from $\begin{pmatrix} P_{k-2}(e) \\ P_{k-2}(e) \otimes \mathbb R^2 \\ P_{k-2}(e) \otimes \mathbb M^2\end{pmatrix}$ to $\mathbb{J}^2(x)\otimes\mathbb M^3$ is  $(a_0,a_1,a_2)\mapsto (a_0(x) \bm e_3^{\otimes 2} + E_3a_1(x)\otimes \bm e_3 + E_3a_2(x) ) \oplus (\partial_{\bm t}a_0(x) \bm e_3^{\otimes 3} + E_3\partial_{\bm t}a_1(x)\otimes \bm e_3^{\otimes 2} + E_3\partial_{\bm t}a_2(x)\otimes \bm e_3 )  \oplus (\partial_{\bm t}^2a_0(x) \bm e_3^{\otimes 4} + E_3\partial_{\bm t}^2a_1(x)\otimes \bm e_3^{\otimes 3} + E_3\partial_{\bm t}^2a_2(x) \otimes \bm e_3^{\otimes 2}).$
\end{itemize}
The remaining mappings in the third column are
\begin{itemize}
    \item [-] (1,3) to (4,3): from  $P_{k-3}(T)\otimes \mathbb T^3$ to $\mathbb J^1(x)\otimes\mathbb M^3$ is $\bm \tau \mapsto \bm \tau(x)\oplus\nabla\bm\tau(x)$;
    \item [-] (2,3) to (4,3): from$\begin{pmatrix} P_{k-3}(f) \otimes \mathbb R^2\\ P_{k-3}(f) \otimes \mathbb R \end{pmatrix}$ to $\mathbb J^1(x)\otimes\mathbb M^3$ is $(a_0,a_1)\mapsto (E_3a_0(x)\otimes \bm e_3 + a_1(x)\bm e_3^{\otimes 2})\oplus (E_3\nabla_fa_0(x)\otimes \bm e_3 + E_3\nabla_fa_1(x)\otimes \bm e_3^{\otimes 2})$.
\end{itemize}

\bibliographystyle{plain}
\bibliography{ref}
\end{document}